\documentclass{amsart}

\numberwithin{equation}{section}

\usepackage[normalem]{ulem}

\usepackage[shortlabels]{enumitem}

\usepackage{amssymb,amsmath,amsthm,verbatim,IEEEtrantools,hyperref}

\theoremstyle{definition}
\newtheorem{theorem}{Theorem}[section]
\newtheorem{thm}[theorem]{Theorem}
\newtheorem*{lemma*}{Lemma}
\newtheorem{lemma}[theorem]{Lemma}
\newtheorem{corollary}[theorem]{Corollary}
\newtheorem{prop}[theorem]{Proposition}

\newtheorem{remark}[theorem]{Remark}

\theoremstyle{definition}

\newtheorem*{Outline-of-the-Rest-of-the-Paper}{Outline of the Rest of the Paper}

\def\forinfmany{\text{ for infinitely many }}
\def\P{P}

\def\A{A}
\def\B{B}
\def\D{D}

\def\RR{\mathbb{R}}
\def\CC{\mathbb{C}}

\def\QQ{\mathbb{Q}}
\def\ZZ{\mathbb{Z}}
\def\NN{\mathbb{N}}

\def\supp{\mathrm{supp}}

\newcommand{\rbr}[1]{\left( {#1} \right)}

\newcommand{\cbr}[1]{ \left\{ {#1} \right\} }
\newcommand{\abr}[1]{\left\langle {#1} \right\rangle}
\newcommand{\abs}[1]{\left| {#1} \right|}
\newcommand{\norm}[1]{\left\|#1\right\|}
\def\one{\mathbf{1}}

\renewcommand{\Re}{\operatorname{Re}}

\newcommand{\OK}{\mathcal{O}_K}
\newcommand{\isom}{\cong}

\newcommand{\prin}{\mathrm{prin}}

\DeclareMathOperator{\tr}{tr}
\DeclareMathOperator{\dist}{dist}

\usepackage{xcolor}

\begin{document}

\author{Robert Fraser \and Kyle Hambrook \and Donggeun Ryou}
\title[Sharpness of the Mockenhaupt-Mitsis-Bak-Seeger theorem]{Sharpness of the Mockenhaupt-Mitsis-Bak-Seeger Fourier restriction theorem in all dimensions}

\thanks{This material is based upon work supported by the National Science Foundation under
Award No. 2400329 and Award No. 2453364.}

\begin{abstract}
We prove the optimality of the exponent in the Mockenhaupt-Mitsis-Bak-Seeger Fourier restriction theorem in all dimensions $d$ and the full parameter range $0 < a,b < d$. Our construction is  deterministic and also yields Salem sets. 
\end{abstract}

\maketitle

\section{Introduction}\label{introduction}
\subsection{Main Result}\label{main result section}
This article concerns the optimality of 
the constant $p_{*}(a,b,d)$ in 
the following general Fourier restriction theorem.

\begin{thm}[Mockenhaupt-Mitsis-Bak-Seeger]\label{frac ST}
Let $0 < a, b < d$, with $d$ an integer. 
Define $p_{*}(a,b,d) = (4d-4a+2b)/b.$
Suppose $\mu$ is a Borel  measure on $\RR^d$ with the following properties: 
\begin{enumerate}[(A)]
    \item $\mu(B(x,r)) 
\lesssim 
    r^{a}$ for all $x \in \RR^d$ and all $r > 0$.  
    \item $|\widehat{\mu}(\xi)| \lesssim 
    (1+|\xi|)^{-b/2}$ for all $\xi \in \RR^d$.  
\end{enumerate}
Then, for each $p \geq p_{*}(a,b,d)$, the following holds: 
\begin{enumerate}[(A)]
\setcounter{enumi}{17}
\item $\|\widehat{f\mu}\|_{L^p} 
\lesssim_p 
\|f\|_{L^2(\mu)}$ for all $f \in L^2(\mu)$.
\end{enumerate}
\end{thm}

\noindent The notation is explained in Subsection \ref{notation}. The history of this theorem and of the question of the optimality of the constant $p_{*}(a,b,d)$ is deferred to Section \ref{motivation}. 

The main result of this article is that the constant $p_{*}(a,b,d)$ in Theorem \ref{frac ST} is best possible for all values of $0 < a,b < d$. 

We prove our main result by constructing a measure   
on a set of vectors that are well-approximable by elements of lattices arising from algebraic number theory. 

Let $K$ be a number field (i.e., a finite extension field of $\QQ$) of degree $d$. 
Let $\OK$ be the ring of integers of $K$.
For each ideal $I \subseteq \OK$, let $N(I)$ denote the norm of $I$ and $I^{-1}$ denote the fractional ideal which is the inverse of $I$. 
Let $B=\cbr{\omega_1,\ldots,\omega_{d}}$ \label{Bbasis} be an integral basis for $K$ (i.e., a basis for $K$ as a $\QQ$-vector space which is also a basis for $\OK$ as a $\ZZ$-module). We identify $\mathbb{Q}^d$ with $K$ by identifying 
$q=(q_1,\ldots,q_n) \in \mathbb{Q}^d$ with 
$q=\sum_{i=1}^{d} q_i \omega_i \in K$. 
Since $B$ is an integral basis, 
this also identifies $\mathbb{Z}^d$ with $\OK$. 
Moreover, if $I \subseteq \OK$ is an integral ideal, $I^{-1}$ is identified with a lattice containing $\mathbb{Z}^d$. 
This identification also extends to lengths and distances:  
If $q=\sum_{i=1}^d q_i \omega_i \in K$ 
and $x = (x_1,\ldots,x_d) \in \RR^d$, 
then 
$|q| = \rbr{\sum_{i=1}^d |q_i|^2}^{1/2}$  
and 
$|q-x| = \rbr{\sum_{i=1}^d |q_i-x_i|^2}^{1/2}$.

Let $\tau > 1$. Define 
$E(K,B,\tau)$ to be the set of all $x \in \RR^d$ such that 
$$
\dist(x, I^{-1}) \leq |N(I)|^{-(\tau+1)/d} 
$$
for infinitely many ideals $I$ of $\OK$.
Here 
$
\dist(x,I^{-1}) = \min\cbr{|x-t|: t \in I^{-1}}. 
$

\begin{thm}[Main Theorem]\label{mainthm3}
Let $1 \leq q < \infty$. 
Let $-d < \rho < d$ 
and define  
$
\rho^+ = \max\cbr{\rho,0}$ and $\rho^- = \max\cbr{-\rho,0}
$. 
There exists a Borel probability measure 
$\mu$ with the following properties: 
\begin{enumerate}[(i)]
    \item $\supp(\mu) \subseteq E(K, B, \tau)$; 
\item $\mu$ satisfies (A) for each exponent $a < \frac{2d  - \rho^-}{1 + \tau}$;
\item 
$\mu$ satisfies (B) for each  exponent $b < \frac{2(d-\rho^+)}{1 + \tau}$; 
\item There exists a sequence of functions $f_k \in L^q(\mu)$ such that
\[\lim_{k \to \infty} \frac{\|\widehat{f_k \mu}\|_{L^p}}{\|f_k\|_{L^q(\mu)}} = \infty\]
whenever $p < p(\tau, \rho, q, d) := \frac{q(d \tau - \rho)}{(q-1) (d-\rho^+)}$.
\end{enumerate}
\end{thm}

Subsection \ref{ideas} describes the main ideas and novel features of the proof of this theorem.  
The formal proof of 
the theorem 
constitutes Sections \ref{Algebraic Number Theory: Basics section} to \ref{sec_Failure_Estimate}.

Theorem \ref{mainthm3} and the following remark 
show that, for any integer $d$ and any real numbers $a$ and $b$ satisfying $0 < a, b < d$, the exponent $p_*(a,b,d)$ in  the Mockenhaupt-Mitsis-Bak-Seeger restriction theorem cannot be improved.

\begin{remark}
We may assume $b \leq 2a$. Indeed, Mitsis \cite[Proposition 3.1]{mitsis-res} 
showed that if (A) and (R) hold, then 
$p \geq 2d/a$.
By combining this with Theorem \ref{frac ST}, we see that (A) and (B) imply $b \leq 2a$. 
Let $p_0 < p_{*}(a,b,d)$. Choose $a_0$ and $b_0$ such that $0 < a < a_0 < d$, $0 < b < b_0 < d$, $b_0 \leq 2a_0$, and $p_0 < p_{*}(a_0,b_0,d) < p_{*}(a,b,d)$. 
If $b_0 \leq a_0$, choose $\tau = \frac{2d}{a_0} - 1$ and $\rho = d\left(1-\frac{b_0}{a_0}\right)$. 
If $a_0 \leq b_0 \leq 2a_0$, choose $\tau = \frac{2d}{b_0} - 1$ and $\rho = 2d\left(\frac{a_0}{b_0}-1\right)$. 
In both cases, $a_0 = \frac{2d  - \rho^-}{1 + \tau}$,  $b_0 = \frac{2(d  - \rho^+)}{1 + \tau}$, and 
$p_{*}(a_0,b_0,d) = p(\tau,\rho,2,d)$. 
Then Theorem \ref{mainthm3} gives a measure $\mu$ satisfying (A) and (B) but failing (R) for $p=p_0$ (and every other $p < p_{*}(a_0,b_0,d)$). 
\end{remark}

\subsection{Deterministic Salem Sets and a Variant of \texorpdfstring{$E(K,B,\tau)$}{E(k,B,tau}.} 

A set in $\RR^d$ is called a Salem set if it supports a Borel measure $\mu$ that satisfies (B) for every value of $b$ less than the Hausdorff dimension of the set. 
There are many random constructions of Salem sets in $\RR^d$ (see \cite{salem}, \cite{kahane-book}, \cite{kahane-1966-fourier}, \cite{kahane-1966-brownian}, \cite{bluhm-1}, \cite{chen-seeger}, \cite{ekstrom}, \cite{LP}, \cite{shmerkin-suomala}). 
The first deterministic Salem sets in $\RR$ of arbitrary dimension were found by 
Kaufman \cite{Kaufman}, who proved 
$$
E(\tau) = \cbr{x \in \RR : |x-r/q| \leq |q|^{-(1+\tau)} \forinfmany (q,r) \in \ZZ \times \ZZ }
$$
is a Salem set of dimension $2/(1+\tau)$ when $\tau > 1$.

Fraser and Hambrook \cite{hambrook-fraser-Rn}  
extended Kaufman's result to $\RR^d$ 
by showing the set 
$$
E_{\prin}(K,B,\tau)=
\cbr{x \in \RR^d : |x-r/q| \leq |q|^{-(1 + \tau)} \text{ for infinitely many } (q,r) \in \OK \times \OK}
$$
is a Salem set of dimension $2d/(1+\tau)$ when $\tau > 1$.

The set $E_{\prin}(K,B,\tau)$ is a variant of $E(K,B,\tau)$ defined in terms of \textit{elements} rather than \textit{ideals}. 
In the present article, we primarily considered $E(K, B, \tau)$ rather than $E_{\prin}(K, B, \tau)$ because 
working with \textit{ideals} rather than \textit{elements} simplifies many parts of the argument, elucidates the underlying structure, and produces better technical results (as elaborated in Subsection \ref{subsec_inno}). 
However, the two sets have many of the same properties, as the next two results illustrate. 

By Frostman's lemma, the exponent $b$ in (B) can never exceed the Hausdorff dimension of the support of $\mu$ (see, e.g.,  \cite[Section 3.6]{mattila-book-2}). 
Thus, Theorem \ref{mainthm3} combined with Proposition \ref{hausdorff upper bound} (which gives an upper bound on the Hausdorff dimension of $E(K,B,\tau)$)
implies

\begin{corollary}\label{Salem_cor}
$E(K,B,\tau)$ is a Salem set of dimension $2d/(1+\tau)$ for every $\tau > 1$.
\end{corollary}

By modifying the proof of Theorem \ref{mainthm3} as described in Section \ref{sec_prin}, we can prove 

\begin{thm}\label{Eprinthm}
The statement of Theorem \ref{mainthm3} also holds for $E_{\prin}(K, B, \tau)$.
\end{thm}

\subsection{Notation}\label{notation}
For a measure $\nu$ on $\RR^d$, the Fourier transform of $\nu$ is   $$\widehat{\nu}(\xi)=\int_{\RR^d} e^{-2\pi i x \cdot \xi } d\nu(x) \quad \text{for all $\xi \in \RR^d$}.$$
We make the following exception to 
this notation: If $F:\RR^d \to \CC$ is $\ZZ^d$-periodic, 
$$
\widehat{F}(s)  
=\int_{[0,1)^d} e^{-2\pi i x \cdot s } F(x) dx \quad \text{for all $s \in \ZZ^d$}.$$

For a measure $\nu$ on $\RR^d$, $L^q(\nu)$ is the space of all measurable function $f:\RR^d \to \RR$ such that $\int |f|^q d\nu < \infty$, with functions that are equal $\nu$-a.e.\ identified.
If $\nu$ is Lebesgue measure, we write $L^q$ instead of $L^q(\nu)$.

We use $B(x,r)$  to denote the closed ball with center $x$ and radius $r$ in the Euclidean norm on $\RR^d$.

For any set $A$, the indicator function of $A$ is the function $\one_A$  defined by $\one_A(x) = 1$ if $x \in A$ and $\one_A(x) = 0$ otherwise. 

A statement of the form ``$F(u) \lesssim G(u)$ for all $u$'' means ``there exists a constant $C>0$ such that $F(u) \leq C G(u)$ for all $u$.'' If the implicit constant $C$ depends on a relevant parameter, say $p$, we write $\lesssim_{p}$ instead of $\lesssim$. 
Note also 
$F(u)  \gtrsim G(u)$ means $G(u) \lesssim F(u)$,   
and 
$F(u) \sim G(u)$ means $F(u) \lesssim G(u)$ and $F(u) \gtrsim  G(u)$.

\subsection{Organization of The Paper}\label{organization}

Section \ref{motivation} contains motivation for Theorem \ref{mainthm3} and discussion about its proof. The proof of Theorem \ref{mainthm3} constitutes sections \ref{Algebraic Number Theory: Basics section} to \ref{sec_Failure_Estimate}.
In Section \ref{Algebraic Number Theory: Basics section}, we introduce the necessary algebraic number theory we need. 
In Section \ref{sec_Single-
Scale}, we give the definitions and prove Fourier estimates for some functions needed for the construction of the measure $\mu$. 
In Section \ref{Intermediate Measures}, we introduce some intermediate measures and prove Fourier-analytic estimates  that will be used in later sections. 
In Section \ref{sec_measure}, we define the measure $\mu$ and show that it satisfies the support and Fourier decay conditions in Theorem \ref{mainthm3}. In particular,  
Proposition \ref{support of mu lemma} and Proposition \ref{Prop_mu_decay}, respectively,  establish (i) and (iii) in Theorem \ref{mainthm3}. 
In Section \ref{regularity_section}, we prove the regularity property in Theorem \ref{mainthm3}. In particular, Proposition \ref{Prop_mu_reg} establishes (ii) in Theorem \ref{mainthm3}.  
In Section \ref{sec_Failure_Estimate}, we prove $\mu$ fails the restriction property (R) for the appropriate values of $p$. In particular,  Proposition \ref{res fail prop} establishes (iv) in Theorem \ref{mainthm3}. 
In Section \ref{Hausdorff Dimension Upper Bound section}, we prove an upper bound on the Hausdorff dimension of $E(K,B,\tau)$. 
Section \ref{sec_prin} contains the modifications to the proof of Theorem \ref{mainthm3} necessary to prove Theorem \ref{Eprinthm}. Finally, Section \ref{sec:glossary} is a glossary of the numerous fixed constants in the paper.

\section{Motivation and Discussion}\label{motivation}

\subsection{Restriction Theorems}\label{restriction theorems}

The problem of restriction of the Fourier transform to small sets is an interesting one with a long history (which we describe only briefly). Because the Fourier transform is an isometry on $L^2$, there is no sensible way to restrict the Fourier transform of an arbitrary $L^2$ function to a set of Lebesgue measure zero. 
On the other hand, because the Fourier transform of a function in $L^1$ is continuous, 
the Fourier transform of an $L^1$ function can be sensibly restricted to any measurable subset of $\mathbb{R}^d$. The area of \textit{Fourier restriction} concerns ways of restricting the Fourier transform of  
functions in spaces besides $L^1$ and $L^2$
to certain sets of Lebesgue measure zero.

We say that the measure $\mu$ on $\RR^d$  satisfies a \textit{Fourier extension} estimate with exponents $p$ and $q$ if
\begin{equation}\label{restriction}
\norm{\widehat{f \mu}}_{L^p} \lesssim_{p,q} 
\norm{f}_{L^q(\mu)}
\text{  for all $f \in L^q(\mu)$.}
\end{equation}
It can be seen using a duality argument (see, e.g., \cite{mattila-book-2}) that the estimate \eqref{restriction} implies that the Fourier transform of functions in $L^{p'}$ can be meaningfully restricted to the support of the measure $\mu$. More precisely, \eqref{restriction} is equivalent to the statement that the Fourier transform map $f \mapsto \widehat{f}$ extends to a bounded linear map from $L^{p'}$ to $L^{q'}(\mu)$. 
Here $p'$ and $q'$ are the conjugate exponents to $p$ and $q$, respectively. 

The foundational result in Fourier restriction theory is the Stein-Tomas theorem \cite{tomas-1975}, \cite{tomas-symp},  which says that \eqref{restriction} holds for $q=2$ and $p \geq 2(d+1)/(d-1)$ when $\mu$ is the surface measure on a sphere in $\RR^d$. Notable generalizations of the Stein-Tomas theorem to the surface measure on quadratic hypersurfaces and arbitrary smooth submanifolds in $\RR^d$ were obtained, respectively, by Strichartz \cite{Strichartz77} and Greenleaf \cite{Greenleaf}. 

All of the restriction theorems mentioned in the previous paragraph are special cases of 
Theorem \ref{frac ST}, 
which is sometimes called the Mockenhaupt-Mitsis-Bak-Seeger restriction theorem. 
Mockenhaupt \cite{mock} and Mitsis \cite{mitsis-res} independently proved Theorem \ref{frac ST} with $p > p_{*}(a,b,d)$;  
the endpoint case $p = p_{*}(a,b,d)$ was established by Bak and Seeger \cite{bak-seeger}.

Condition (A) in Theorem \ref{frac ST} is a well-known condition in geometric measure theory sometimes called the Frostman condition or ball condition. 
Mitsis \cite{mitsis-res} proved that condition (B) in Theorem \ref{frac ST} implies condition (A) with $a=b/2$. 
This makes Theorem \ref{frac ST} quite remarkable ---  it lets one use the pointwise estimate (B) on $\widehat{\mu}$ to conclude an $L^p$ decay estimate (R) on $\widehat{f\mu}$ for all functions $f \in L^2(\mu)$. 

Carnovale, Fraser, and de Orellana \cite{carnovale2025l2restrictionestimatesfourier} recently generalized the Mockenhaupt-Mitsis-Bak-Seeger restriction theorem by replacing condition (B) with a condition on the Fourier spectrum of $\mu$.

\subsection{Previous Work on the Sharpness of Theorem \ref{frac ST}.}

Mockenhaupt \cite{mock} and Mitsis \cite{mitsis-res} 
both noted the question of whether the exponent $p_*(a,b,d)$ in Theorem \ref{frac ST} is best possible was an open problem.

When $d \geq 2$ and $a=b=d-1$, the exponent $p_{*}(a,b,d)$ in Theorem \ref{frac ST} is known to be optimal because of an example due to Knapp (according to  \cite{tomas-symp}) for the surface measure on the sphere. 
The original purpose of Knapp's example was to show that the exponent $2(d+1)/(d-1)$ in the Stein-Tomas theorem (mentioned in Subsection \ref{restriction theorems}) is best possible. 
The Knapp example was later extended 
 to smooth hypersurfaces with nonvanishing Gaussian curvature  by Strichartz \cite{Strichartz77}. See also Example 1.8 of Demeter \cite{Demeter} for a further generalization of Knapp's example  that shows \eqref{restriction} cannot hold if $\mu$ is a measure on a smooth hypersurface and $p < \frac{q(d + 1)}{(q - 1) (d - 1)}.$

 Hambrook and {\L}aba \cite{HL2013}  proved the optimality of $p_*(a, b, d)$ in Theorem \ref{frac ST} when $d=1$ and $0 < b \leq a \leq d$. 
 (In fact, Hambrook and {\L}aba \cite{HL2013} only claimed the case $a=b$ for $a,b$ of a special form.   
Hambrook \cite{hambrook-thesis} 
showed how optimality for all $b \leq a$ could be deduced from the result of \cite{HL2013}.  
Chen \cite{chen} modified the construction of 
\cite{HL2013} to obtain a technically stronger result that also implied optimality for $b \leq a$.)

By combining 
the ideas of \cite{HL2013} and Knapp's example, 
Hambrook and {\L}aba \cite{HL2016} proved the optimality of $p_*(a,b,d)$ for $d \geq 2$ and $d-1 \leq b \leq a \leq d$.  
Note that  none of the sharpness results mentioned above covered the case $b>a$. Fraser, Hambrook, and Ryou 
\cite{fraser-hambrook-ryou} 
built upon the idea  of  \cite{HL2013} to
prove the optimality of $p_*(a,b,d)$ when $d=1$ in the full range $0 < a,b < d$ (notably, the construction is  deterministic). 
Li and Liu \cite{bochen-dimension} independently obtained the same optimality result for $d=1$ by a similar construction. 
The present article 
fully resolves 
the problem of the optimality of $p_*(a,b,d)$ for \textit{all} $d \geq 1$ and the full range $0 < a,b < d$.

\begin{remark}
We have considered the question of the sharpness of $p_*(a,b,d)$ in Theorem \ref{frac ST} by fixing $a,b,d$ and asking  whether the statement holds with $p_*(a,b,d)$ replaced by a smaller number. An alternative way to consider the question  is to look for measures which satisfy (A), (B), and (R) for $p < p_{*}(a,b,d)$.  For some results in this vein, see Shmerkin and Suomala \cite{shmerkin-suomala}, Chen and Seeger \cite{chen-seeger}, {\L}aba and Wang \cite{laba-wang}, and Ryou \cite{Ryou-2023}. 
\end{remark}

\subsection{Development and Key Ideas of the Proof of Theorem \ref{mainthm3}}\label{ideas}

The key to failure of restriction estimates in Knapp's example is \textit{constructive interference}. 
The smoothness of hypersurfaces guarantees that a $\delta$-neighborhood of a point on the hypersurface will be contained in a $\delta \times \cdots \times \delta \times \delta^2$-rectangle oriented along the tangent plane to the hypersurface. 
The uncertainty principle (or a direct calculation) implies that 
the Fourier transform of the indicator function of this $\delta$-ball should therefore be $\gtrsim \delta^{d-1}$ in a dual rectangle with side lengths $\delta^{-1} \times \cdots \times \delta^{-1} \times \delta^{-2}$. 

When studying restriction problems for \textit{fractal sets}, the easiest way to achieve constructive interference is to choose the indicator function of a neighborhood of an arithmetic progression. For simplicity, we will consider the case $d = 1$. In fact, the indicator function of a $\delta$-neighborhood of an arithmetic progression with common difference $1/z$ will have large Fourier transform near frequencies of the form $\xi = n z, n \in \mathbb{Z}, |n| \leq \delta^{-1}$. Thus we want to build a measure whose support contains relatively long arithmetic progressions at infinitely small scales. Since we also want a construction producing measures for a large range for $a$ and $b$, we are naturally inclined to consider constructions of Salem sets.

These observations are used by Hambrook and {\L}aba \cite{HL2013} and a later work by Chen \cite{chen} to provide a random Cantor measure for which the exponent in (R) cannot be improved beyond $p_*(a,b,d)$. In more detail, \cite{HL2013} and \cite{chen} start from a construction by {\L}aba and Pramanik \cite{LP2009} of Salem sets by a random Cantor-type method. \cite{HL2013} and \cite{chen} insert a deterministic arithmetic progression at each stage, and use this to obtain the optimality of $p_*(a,b,d)$.

There are many other random constructions of Salem sets besides that of {\L}aba and Pramanik \cite{LP2009},  but it is not known how to build the necessary deterministic progressions into these other constructions.

Unfortunately, the {\L}aba-Pramanik construction of Salem sets does not extend to $d>1$. Thus we cannot prove the optimality of $p_{*}(a,b,d)$ in the same way. Moreover, the random constructions of Salem sets for $d>1$ do not appear amenable to the insertion (or location) of the desired deterministic arithmetic progressions. 

The path around this obstacle came from studying the problem of optimality of $p_{*}(a,b,d)$ when $d=1$ with the additional constraint that the measure be deterministic (a problem proposed to us by Andreas Seeger \cite{seeger-private}.)

In order to construct deterministic examples of fractal sets allowing for constructive interference, 
it is useful to have an example of a set supporting measures satisfying (A) and (B) for 
prescribed exponents $a$ and $b$ that also contains long arithmetic progressions. A natural choice is the set $E(\tau)$, for three reasons. 
First, Kaufman \cite{Kaufman} constructed a measure satisfying (B) with $b$ maximal, and Papadimitropoulos \cite{pa-thesis} modified this construction so that (A) is satisfied with maximal $a$. 
Second, it is easy to locate a compact subset $F$ of $E(\tau)$ for which fattenings of $F$ contain long arithmetic progressions.  
Third, the limsup structure of the set $E(\tau)$ allows for the insertion of arithmetic progressions of different length if necessary (it turns out to be necessary in the case where $a \neq b$).

In \cite{fraser-hambrook-ryou}, Fraser, Hambrook, and Ryou 
modified Kaufman's construction (and Papadimitropoulos's construction \cite{pa-thesis}) so that the measure satisfied (A) and (B) for prescribed values of $a$ and $b$ and failed $(R)$ for $p<p_*(a,b,d)$. As alluded to above, inserting (and modifying the weight on) a generalized arithmetic progression is the key to establishing the failure of (R). 

With this deterministic problem solved, the path to $d>1$ became clear because 
Kaufman's construction of Salem measures on $E(\tau)$ \textit{does} generalize to $d>1$. 
In fact, several year earlier, Hambrook and Fraser \cite{hambrook-fraser-Rn} constructed a measure on $E_{\prin}(K,B,\tau)$ 
with maximal Fourier decay.

Indeed, it is natural to seek measures on the set $E_{\prin}(K,B,\tau)$ as possible counterexamples to (R) for $p < p_*(a,b,d)$. 
First, the set $E_{\prin}(K,B,\tau)$ 
supports measures satisfying (A) and (B) for prescribed  $a$ and $b$; 
second, it is easy to locate a compact subset of $E_{\prin}(K,B,\tau)$ is naturally concentrated around a small neighborhood of a small number of lattices; third, because $E_{\prin}(K,B,\tau)$ is a limsup set, it is easy to insert a generalized arithmetic progression of a desired length in order to obtain a counterexample if $a \neq b$.

We chose to use the set $E(K,B,\tau)$ (defined using ideals) rather than $E_{\prin}(K,B,\tau)$ (defined using elements) for our proof that $p_*(a,b,d)$ is optimal. The set $E(K,B,\tau)$ has the same advantageous properties as $E_{\prin}(K,B,\tau)$ described above. 
However, using the ideal framework of $E(K,B,\tau)$ leads to a proof that makes the fundamental ideas more transparent and produces better technical results. Details on these improvements are discussed in Subsection \ref{subsec_inno}.

The proof of Theorem \ref{mainthm3} 
synthesizes  
the Fraser-Hambrook \cite{hambrook-fraser-Rn} construction 
using the ideas of 
\cite{fraser-hambrook-ryou}. 
The method used to establish the necessary Fourier decay estimate in Theorem \ref{mainthm3} is similar to the method in  \cite{hambrook-fraser-Rn}, though some adjustments are necessary because $E(K,B,\tau)$ in Theorem \ref{mainthm3}  is defined using ideals while the set $E_{\prin}(K,B,\tau)$ in \cite{hambrook-fraser-Rn} is defined using elements. The additional abstraction leads to a sharper and more technically sound version of the method in \cite{hambrook-fraser-Rn}. While the argument is more difficult (due to its generality), it is more efficient.

The regularity property in Theorem \ref{mainthm3} was not considered in \cite{hambrook-fraser-Rn}. 
The argument in the present paper extends the regularity result of Papadimitropoulos \cite{pa-thesis} to the number field setting. 
The key ingredient is the separation lemma (Lemma \ref{separation lemma}) and its application in Section \ref{regularity_section}. 
As in \cite{fraser-hambrook-ryou} (and in contrast to \cite{pa-thesis}) we use a Fourier-analytic argument in proof of regularity (see Step 2 in the proof of Lemma \ref{k-ball regularity}). 
The failure of the estimate (R) is proved by an argument similar to that used in \cite{fraser-hambrook-ryou}, though it requires more advanced ideas due to the more general setting. The role of the arithmetic progression at each scale is played by the inverse of a prime ideal. 
As in \cite{fraser-hambrook-ryou}, we use an argument relying on the Chinese remainder theorem. However, we use a ring-theoretic version of the theorem and argument here substantially generalizes the integer-congruence argument in \cite{fraser-hambrook-ryou}.

\subsection{Innovations on Earlier Work}\label{subsec_inno}

This article innovates on the authors' earlier work on (1) construction deterministic construction of Salem sets and (2) sharpness of the Mockenhaupt-Mitsis-Bak-Seeger restriction theorem. The two research programs were related, but used substantially different methods.  
The present article represents a \textit{synthesis} of these two programs. 

Indeed, 
the authors' work on constructing Salem sets (\cite{hambrook-R2},\cite{hambrook-fraser-Rn})
is based on well-approximation of vectors by elements of number fields and relies heavily on algebraic number theory. In contrast, the authors' early work (\cite{HL2013}, \cite{HL2016}) on the sharpness of the Mockenhaupt-Mitsis-Bak-Seeger theorem was based on a random Cantor set construction and used  probabilistic methods but no number theory at all.  
Moreover, while the authors' most recent work \cite{fraser-hambrook-ryou} on the sharpness problem is based on well-approximation of real numbers by rationals, the tools used are only elementary number theory in the integers -- no algebraic number theory is involved.

This article primarily combines the methods employed in \cite{fraser-hambrook-ryou} for constructing counterexamples to restriction estimates in $\mathbb{R}$ with the methods used in \cite{hambrook-fraser-Rn} to construct Salem sets in $\mathbb{R}^n$. 

The fundamental strategy employed in \cite{fraser-hambrook-ryou} is to construct the 
measure $\mu$ in a way so that $\mu$ could be decomposed (at 
infinitely many different scales) into a piece $\mu_1$ that has small support, but interferes constructively with itself, 
and a piece $\mu_2$, which is has large support large but behaves independently 
from $\mu_1$. Specifically, $\mu_1$ is chosen to be supported in a neighborhood of an arithmetic progression. 
It is convenient for $\mu_2$ to be chosen to be a Salem-like measure in order to provide the best control over the 
exponents $a$ and $b$ appearing in the assumptions (A) and (B). The length of the progression determining 
the support of $\mu_1$, together with the approximation exponent used to define $\mu_2$, will dictate the behavior of the measure $\mu$. Then, if the function $f$ 
is chosen so that its support is adapted to the support of $\mu_1$, then $\widehat{f \mu_1}$ will 
be large on a large set because of constructive interference, and the independence of $\mu_1$ and $\mu_2$ guarantees that $\widehat{f \mu}$ is close to $\widehat{f \mu_1}$ on this set.

In order to adapt this construction to higher dimensions, it is natural to take the measure $\mu_1$ to be supported on a neighborhood of a \textit{lattice} 
in $\mathbb{R}^n$, since the arithmetic structure provided by a lattice can be used to establish the required constructive interference. Since the measure $\mu_2$ should be a Salem-like measure, 
we would like to choose this measure to be the one from the Salem set construction appearing in \cite{hambrook-fraser-Rn}. In that paper, the measure $\mu_2$ is constructed from lattices arising as multiples of numbers of the form $\frac{1}{q}$, where $q$ is an algebraic integer. 
It is natural to choose the lattice 
for $\mu_1$
to also consist of multiples of the reciprocal of an algebraic integer.

However, if this is to be done directly, there is some additional difficulty in controlling the interference between $\mu_1$ and $\mu_2$. 
In the one-dimensional setting, this interference is controlled using a simple argument involving unique factorization and the Chinese remainder theorem. 
However, algebraic number fields do not exhibit unique factorization, so the methods from the one-dimensional setting cannot directly be applied.

While this difficulty could be overcome using the finiteness of the ideal class group, we adopt a more elegant approach. Namely, we modify the choice of $\mu_2$ so its support is contained in a union of neighborhoods of \textit{inverses of ideals} rather than being contained in a union of neighborhoods \textit{multiples of elements of the number field}. 
This ideal-based approach generalizes the element-based and allows us to apply the methods \cite{fraser-hambrook-ryou} cleanly in the higher-dimensional setting. 
Moreover, since we use inverses of \textit{prime} ideals, unique factorization is recovered
so that the methods of \cite{hambrook-fraser-Rn} and \cite{fraser-hambrook-ryou} and can be applied 
efficiently.

This approach naturally synthesizes the techniques from our previous two articles, and provides a new perspective on the one-dimensional result \cite{fraser-hambrook-ryou}. The ideal-based approach also provides some technical advantages over the result in the article \cite{hambrook-fraser-Rn}, which we will now describe. First, the dot product $x \cdot y$ is replaced with the trace form $\tr_{K/\mathbb{Q}}(xy)$. This is a natural perspective to take in the setting of algebraic number theory because of the basis-independence of this form on elements $x, y \in K$. Second, the set of ideals $Q(M)$ consists of prime ideals $I$ with $M^d/2 \leq N(I) \leq M^d$, which is more natural than the set of denominators considered in the article \cite{hambrook-fraser-Rn}. Third, the use of prime ideals gives bounds that lose only powers of logarithms rather than divisor-bound type losses.

\section{Algebraic Number Theory}\label{Algebraic Number Theory: Basics section}

In this section, we give the definitions and facts from algebraic number theory necessary for the proof of our main theorem. As a reference for this material, see \cite{Jarvis14}. 
We also prove some lemmas that we will need.

Recall that $K$ is a number field of degree $d$, $\OK$ is its ring of integers, and $B = \cbr{\omega_1,\ldots,\omega_d}$ \label{Bbasis-2} is an integral basis for $K$.

Let $\mathcal{E}_K$ denote the set of all injective homomorphisms 
from $K$ into $\mathbb{C}$. There are precisely $d$ such homomorphisms. 
For each $q \in K$, the field norm of $q$ is 
\begin{align}\label{field norm formula}
N(q) 
=
\prod_{\tau \in \mathcal{E}_K} \tau(q).
\end{align}
and the field trace of $q$ is 
\begin{align}\label{field trace formula}
\tr(q) 
=
\sum_{\tau \in \mathcal{E}_K} \tau(q).
\end{align}
The field norm is multiplicative: 
$N(ab)=N(a)N(b)$ for all $a,b \in K$. 
The field trace is additive: 
$\tr(a+b)=\tr(a) + \tr(b)$ for all $a,b \in K$.

\begin{lemma}\label{C1 bounds lemma}
For the integral basis $B = \cbr{\omega_1,\ldots,\omega_d}$, define 
\begin{align*}
C_B = \max\cbr{ \textstyle \rbr{\sum_{i=1}^d |\tau(\omega_i)|^2}^{1/2} : \tau \in \mathcal{E}_K }.
\end{align*}
For every $q \in K$, 
\begin{align}\label{C1 norm bound}
|N(q)| \leq C_B^d |q|^d.
\end{align}
\end{lemma}
\begin{proof}
If $q = \sum_{i=1}^d q_i \omega_i \in K$ and $\tau \in \mathcal{E}_K$, then 
$\tau(q) = \sum_{i=1}^d q_i \tau(\omega_i) \in \CC$.  
The Cauchy-Schwarz inequality implies $|\tau(q)| \leq C_B|q|$, so that \eqref{field norm formula} implies \eqref{C1 norm bound}.
\end{proof}

\begin{lemma}
For all $a,b \in K$, $|ab| \lesssim_B |a||b|$.
\end{lemma}
\begin{proof}
Write $a = \sum_{i=1}^{d} a_i \omega_i$, 
$b = \sum_{j=1}^{d} b_j \omega_j$, and $\omega_i \omega_j = 
\sum_{k=1}^{d} c_{ijk} \omega_k$, where $a_i, b_j, c_{ijk} \in \QQ$. 
By the Cauchy-Schwarz inequality, 
$$
|ab|^2 
= 
\sum_{k}  \abs{\sum_i \sum_j a_i b_j c_{ijk} }^2
\lesssim_B \rbr{\sum_i |a_i|^2} \rbr{\sum_j |b_j|^2}
= |a|^2|b|^2. 
$$
\end{proof}

An ideal of $\OK$ is an additive subgroup $I$ of $\OK$ such that $rI = \cbr{ra:a \in I} \subseteq I$ for all $r \in \OK$. A fractional ideal of $\OK$ is a subset $J \subseteq K$ of the form 
$$J = \frac{1}{r}I = \cbr{\dfrac{1}{r}a : a \in I},$$
where $I$ is an ideal of $\OK$ and $r$ is a non-zero element of $\OK$  (neither $I$ nor $r$ are uniquely determined by $J$).  
Every ideal of $\OK$ is a fractional ideal of $\OK$, and every fractional ideal of $\OK$ is an $\OK$-submodule of $K$. 
A fractional ideal of $\OK$ is an ideal of $\OK$ if and only if it is contained in $\OK$. The principal fractional ideal generated by the element $q \in K$ is $\abr{q} = q\OK$. 

The product of two fractional ideals $I$ and $J$ is the fractional ideal
$$
IJ = \cbr{\sum_{i=1}^{n} a_i b_i : a_i \in I, b_i \in J, n \in \NN}. 
$$
The non-zero fractional ideals of $\OK$ form a commutative group under multiplication.  The multiplicative identity is $\OK$. The inverse of a non-zero fractional ideal $J$ is 
$$
J^{-1} = \cbr{a \in K : aJ \subseteq \OK }. 
$$
The group is partially ordered by inclusion
: $I \subseteq J$ implies $HI \subseteq HJ$ for all non-zero fractional ideals $H,I,J$.  
Consequently, if $I \subseteq \OK \subseteq J$, then $J^{-1} \subseteq \OK \subseteq I^{-1}$.

A non-zero proper ideal $P$ of $\OK$ is called prime if 
$ab \in \OK \setminus P$ whenever $a,b \in \OK \setminus P$. 
It is not generally true that every non-zero, non-unit element of $\OK$ can be uniquely factored into a product of prime elements. 
However, every non-zero proper ideal of $\OK$ can be uniquely factored 
as a product of prime ideals.

If $J = \frac{1}{r}I$ is a fractional ideal of $\OK$, where $I$ is an ideal of $\OK$ and $r$ is a non-zero element of $\OK$, the ideal norm of $J$ is defined to be 
\begin{align*}
N(J) = \dfrac{|\OK/I|}{|\OK/\abr{r}|},    
\end{align*}
which is finite whenever $J$ is non-zero. 
It can be shown that this definition does not depend on the choice of $I$ and $r$. Moreover, the ideal norm is multiplicative: 
For all fractional ideals $J_1,J_2$ of $\OK$, 
$$N(J_1 J_2) = N(J_1)N(J_2).$$

\begin{lemma}\label{idealquotient}
If $I_2 \subseteq I_1$ are fractional ideals of $\OK$,  
then  $N(I_2) = N(I_1)|I_1/I_2|$. 
\end{lemma}
\begin{proof}
Choose $r \in \OK \setminus \cbr{0}$ such that $rI_2 \subseteq rI_1 \subseteq \OK$. 
Then 
\begin{align*}
N(\abr{r}) N(I_2)
&=
N(r I_2)
=
|\OK/rI_2|
=
|\OK/rI_1| \cdot |rI_1/rI_2| \\
&=
N(r I_1) |rI_1/rI_2|
=
N(r I_1) |I_1/I_2| 
=
N(\abr{r}) N(I_1) |I_1/I_2|. 
\end{align*}
\end{proof}

The field norm and ideal norm are related as follows:
For all $q \in K$, 
\begin{align}\label{field-norm-ideal-norm-formula}
|N(q)| = N(\abr{q}).  
\end{align}
This fact combines with Lemmas \ref{C1 bounds lemma} and \ref{idealquotient} to give the following separation lemma. 

\begin{lemma}\label{separation lemma}
Let $I_1$ and $I_2$ be non-zero fractional ideals of $\OK$. Let $r_1 \in I_1^{-1}$ and $r_2 \in I_2^{-1}$ with $r_1 \neq r_2$. Then 
$$
C_B|r_1 - r_2| 
\geq (N(I_1 \cap I_2))^{-1/d}
\geq (N(I_1) N(I_2))^{-1/d}. 
$$
\end{lemma}

\begin{proof}
Since $I_1 I_2 \subseteq I_1 \cap I_2$, Lemma \ref{idealquotient} implies $N(I_1 I_2) \geq N(I_1 \cap I_2)$, which gives the second inequality. Now we prove the first inequality.
We have $r_1 (I_1 \cap I_2) \subseteq r_1 I_1 \subseteq \OK$ 
and 
$r_2 (I_1 \cap I_2) \subseteq r_2 I_2 \subseteq \OK$. 
Thus $\abr{r_1 - r_2}(I_1 \cap I_2) \subseteq \OK$. Now Lemma \ref{idealquotient} implies 
$N(\abr{r_1 - r_2})N(I_1 \cap I_2) \geq N(\OK) = 1$. 
The desired result follows from Lemma \ref{C1 bounds lemma} and \eqref{field-norm-ideal-norm-formula}.
\end{proof}

We now define the different ideal $\delta$ of $\OK$.  
The set 
$$
\{ x \in K : \tr(xy) \in \ZZ \ \text{ for all } y \in \OK \}
$$ 
is a fractional ideal of $\OK$. 
The different ideal $\delta$ \label{different ideal defn} is the \textit{inverse} of this fractional ideal. In other words, $\delta^{-1}$ is equal to the fractional ideal above. Note that $\delta \subseteq \OK \subseteq \delta^{-1}$.

Recall that the integral basis is $B = \cbr{\omega_1,\ldots,\omega_d}$. 
Define the $d \times d$ matrix \label{Tmatrix}
$$T =
[\tr(\omega_i \omega_j)].$$ 
Since $(x,y) \mapsto \tr(xy)$ is a non-degenerate bilinear form on $\QQ^d \isom K$, $T$ is invertible and 
the dot product is related to the trace by 
$$\tr(r(T^{-1}s)) = r \cdot s \text{ for all $r,s \in \QQ^d \isom K$}.$$
It follows that $T^{-1}(\ZZ^d) = \delta^{-1}$.

For $u \in \RR$, define 
$$
e(u)=e^{-2\pi i u}. 
$$

\begin{lemma}\label{NEWGeomSeriesProp}
Let $I$ be a non-zero ideal in $\OK$ and let $R(I) \subseteq I^{-1}$ be a complete set of representatives of $ I^{-1} / \OK$. Let $s \in \OK \isom \ZZ^d$. Then 
\begin{align*}
\sum_{r \in R(I)} e(r \cdot s)
=
\left\{
\begin{array}{ll}
N(I) & \text{if $s \in T(\delta^{-1} I) $}  \\
0       & \text{if $s \notin T(\delta^{-1} I) $.} 
\end{array}
\right.
\end{align*}
\end{lemma}

\begin{proof}
For any fractional ideal $J$, we have $J=\OK J$, and hence (as is easily verified)  $J \subseteq \delta^{-1}$ if and only if $\tr(z) \in \ZZ$ for all $z \in J$. 
Taking 
$J = (T^{-1} s) I^{-1}$, it follows that $s \in T(\delta^{-1}I)$ if and only if $r \cdot s =  \tr(r (T^{-1} s)) \in \ZZ$ for all $r \in I^{-1}$.

Therefore, if $s \in T(\delta^{-1} I)$, we have  
\[
\sum_{r \in R(I)} e(r \cdot s) = \sum_{r \in R(I)} 1 = |R(I)|.
\]
But, by Lemma \ref{idealquotient}, 
\[
|R(I)| = |I^{-1} / \OK| 
= \dfrac{N(\OK)}{N(I^{-1})} 
= N(I).
\]

Now assume $s \notin T(\delta^{-1} I)$. 
Then there exists an $r_0 \in I^{-1}$ such that $r_0 \cdot s = \tr(r_0(T^{-1}s)) \notin \ZZ$, and hence $e(r_0 \cdot s) \neq 1$. 
By pigeonholing, $r_0 + R(I)$ is also a complete representatives of $I^{-1}/\OK$.
So there is a bijection $f:R(I) \to r_0 + R(I)$ such that $r - f(r) \in \OK \isom \ZZ^d$ for all $r \in R(I)$. 
Since $r \mapsto e(r \cdot s)$ is $\ZZ^d$-periodic, 
\[
\begin{split}
\sum_{r \in R(I)} e(r \cdot s)  
= \sum_{r \in R(I)} e(f(r) \cdot s)
= e(r_0 \cdot s) \sum_{r \in R(I)} e(r \cdot s).
\end{split}
\]
Since $e(r_0 \cdot s) \neq 1$, 
the sum must be zero.
\end{proof}

We shall need the following version of the Chinese remainder theorem. 

\begin{lemma}\label{concrete crt lemma}
Suppose $I_1$, $I_2$, and $D$ are fractional ideals of $\OK$. 
Suppose $a_1, a_2 \in D$. 
\begin{enumerate}[(a)]
\item If $a_1 - a_2 \notin D(I_1 + I_2)$, then
\[(a_1 + DI_1) \cap (a_2 + DI_2) = \emptyset.\]
\item If $a_1 - a_2 \in D(I_1 + I_2)$, then there exists $a \in D$ such that 
$$
(a_1 + DI_1) \cap (a_2 + DI_2) = a + (D I_1 \cap DI_2).
$$
\end{enumerate}
\end{lemma}

\begin{proof}
\mbox{}\\
\begin{enumerate}[(a)]
    \item 
If $(a_1 + D I_1) \cap (a_2 + D I_2) \neq \emptyset$, then $a_1 + d_1 = a_2 + d_2$ for some $d_1 \in DI_1$ and $d_2 \in DI_2$, and so 
$$
a_1 - a_2 = d_2 - d_1 \in DI_1 + DI_2 = D(I_1 + I_2). 
$$
\item 
Note 
$$
(I_1 + I_2)^{-1}I_1 + (I_1 + I_2)^{-1}I_2 = (I_1 + I_2)^{-1} (I_1 + I_2) = \OK.$$ 
So there exist $x_1 \in (I_1 + I_2)^{-1}I_1$ and $x_2 \in (I_1 + I_2)^{-1}I_2$ such that $x_1+x_2=1$. 
Define $a = x_1a_2 + x_2a_1$. 
Suppose $a_1 - a_2 \in D(I_1 + I_2)$. 
Then 
\begin{align*}
a &= x_1a_2 + (1-x_1)a_1 = a_1 + (a_2 - a_1)x_1 
\\
&\in
a_1 + D(I_1 + I_2)(I_1 + I_2)^{-1}I_1
=
a_1 + DI_1. 
\end{align*}
A similar calculation shows that $a \in a_2 + DI_2$. The desired equality is now easily verified.
\end{enumerate}
\end{proof}

\section{Intermediate  Functions}\label{sec_Single-
Scale}

In this section, we define several functions necessary for the construction of the measure $\mu$. We also prove necessary Fourier-analytic estimates on these functions.

\subsection{The Functions \texorpdfstring{$\phi$}{phi} and \texorpdfstring{$\Phi$}{Phi}}
\label{functions phi and Phi section}

Fix $\phi: \RR^d \to \RR$ \label{phi defn} such that $\phi$ is $C^{\infty}$, $\phi \geq 0$, $\supp(\phi) \subseteq [-1,1]^d$, and 
$\widehat{\phi}(0) = \int \phi(x) dx = 1$. For each non-zero ideal $I$ of $\OK$ and each $\eta > 0$, define 
\begin{align}\label{Phi defn}
\Phi_{I,\eta}(x) = \sum_{r \in I^{-1}} \eta^{-d} \phi((x-r)/\eta). 
\end{align}
On any given bounded subset of $\RR^d$, only finitely many terms of the sum are non-zero (because $\supp(\phi)$ is compact and $I^{-1}$ is discrete). 
Note that $\eta^{-d} \phi ((x-r)/\eta)$ is a bump function for an $\eta$-neighborhood of $r$. 
Thus $\Phi_{I,\eta}$ is a bump function for a $\eta$-neighborhood of the lattice $I^{-1}$. 
Note also that $\Phi_{I,\eta}$ is $C^{\infty}$, non-negative, and $\ZZ^d$-periodic.

\begin{lemma}\label{Phi-Fourier-Lemma}
Let $I$ be a non-zero ideal of $\OK$, let $s \in \ZZ^d \isom \OK$, and let 
$\eta > 0$. Then 
\begin{equation}\label{Phi-Fourier-eqn-1}
\widehat{\Phi_{I,\eta}}(s) = \widehat{\phi}(\eta s) N(I)\mathbf{1}_{\delta^{-1}I} (T^{-1} s). 
\end{equation}
Moreover, if $0 < \eta < 2^{-1}d^{-1/2} C_B^{-1}N(I)^{-1/d}$, then  
\begin{equation}\label{Phi-Fourier-eqn-2}
\widehat{\Phi_{I,\eta}^2}(s) = \eta^{-d} \widehat{\phi^2}(\eta s) N(I)\mathbf{1}_{\delta^{-1}I} (T^{-1} s).
\end{equation}
\end{lemma}
\begin{proof}
For each non-zero ideal $I$, set $R(I) = I^{-1} \cap [0,1)^d$. Note that $R(I)$ is a complete set of representatives of $I^{-1} / \OK$. 
Every element $r \in I^{-1}$ can be written uniquely as $r=r'+k$, where $r' \in R(I)$ and $k \in \OK$. Furthermore, $\OK$ is identified with $\ZZ^d$ via the integral basis $B$. Thus 
\begin{align*}
\widehat{\Phi_{I,\eta}}(s)
&=
\int_{[0,1)^d} \sum_{r \in I^{-1}}  \eta^{-d} \phi ((x-r)/\eta) e(s \cdot x)  dx
\\
&=
\sum_{r \in I^{-1}}  \int_{[0,1)^d}  \eta^{-d} \phi ((x-r)/\eta) e(s \cdot x)  dx
\\
&=
\sum_{r \in R(I)} \sum_{k \in \ZZ^d}  \int_{[0,1)^d}  \eta^{-d} \phi ((x - r - k)/\eta) e(s \cdot x)  dx
\\
&=
\sum_{r \in R(I)} \int_{\RR^d}  \eta^{-d} \phi ((x - r)/\eta) e(s \cdot x)  dx
\\
&=
\sum_{r \in R(I)}  e(s \cdot r) \int_{\RR^d}  \phi (u) e(\eta s \cdot u)  du
\\ 
&= 
\widehat{\phi} (\eta s) \sum_{r \in R(I)} e(s \cdot r).
\end{align*}
Now Lemma \ref{NEWGeomSeriesProp} implies \eqref{Phi-Fourier-eqn-1}. 
If $0 < \eta < 2^{-1}d^{-1/2} C_B^{-1}N(I)^{-1/d}$, then Lemma \ref{separation lemma} implies $\phi((x-r)/\eta)$ and $\phi((x-r')/\eta)$ have disjoint support whenever $r$ and $r'$ are distinct elements of $I^{-1}$, and so 
$$
\Phi_{I,\eta}^2(x) = \sum_{r \in I^{-1}} \sum_{r' \in I^{-1}} \eta^{-2d} \phi((x-r)/\eta) \phi((x-r')/\eta) = \sum_{r \in I^{-1}} \eta^{-2d} \phi^2((x-r)/\eta). 
$$ 
Thus an argument similar to the one above gives \eqref{Phi-Fourier-eqn-2}.
\end{proof}

\subsection{The collection of ideals \texorpdfstring{$Q(M)$}{Q(M)}}
\label{collection of ideals section}

For each real number $M > 1$, 
define $Q(M)$ to be the set of prime ideals $I$ of $\OK$ such that $M^d/2 \leq N(I) \leq M^d$. 
By the Landau prime ideal theorem, 
\begin{equation}\label{QMsize}
|Q(M)| \sim \frac{M^d}{\log(M)}.
\end{equation}

\begin{lemma}\label{primedivisorbound}
Let $J$ be a non-zero integral ideal of $\OK$. Then 
\begin{equation}\label{primedivisorineq}
|\cbr{I \in Q(M): J \subseteq I}| \leq \frac{\log N(J)}{ \log (M^d/2)}.
\end{equation}
\end{lemma}
\begin{proof}
Since $J$ factors uniquely as a product of prime ideals $I$ that contain $J$, and since the ideal norm is multiplicative, 
\[\prod_{\substack{I \in Q(M) \\ J \subseteq I}} N(I) \leq 
N(J).\] 
The desired result follows by taking logarithms and using that each ideal $I \in Q(M)$ satisfies $\log (M^d/2) \leq \log N(I)$. 
\end{proof}

\subsection{The Functions \texorpdfstring{$F_k$}{Fk}}
\label{functions Fk section}

Let $(M_k)_{k=0}^{\infty}$ \label{Mk defn}
be an increasing sequence of positive real numbers. We assume $M_0$ is large 
and the sequence $(M_k)_{k=0}^{\infty}$ increases rapidly. Sufficient size and growth conditions will be noted as they arise in the argument. 

Fix $-d < \rho < d$. For each 
$k$, fix an ideal $J_k \in Q(M_k^{1+\frac{\rho}{d}})$. \label{Jk defn} 
Thus $J_k$ is a prime ideal  
such that 
$$M_k^{d+\rho}/2 \leq N(J_k) \leq M_k^{d+\rho}.$$ 
Assuming $M_1$ is large enough, \eqref{QMsize} guarantees such an ideal exists. Note $J_k \in Q(M_k)$ if $\rho = 0$. 

Recall that $\rho^{+}=\max\cbr{\rho,0}$ and 
$\rho^{-}=\max\cbr{-\rho,0}$. 
Define 
$$
\label{Pk defn}
\P_k = 
\left\{
\begin{array}{ll}
M_k^{\rho^-}
& \text{if } k \text{ is odd}, \\
0 & \text{if } k \text{ is even}. \\
\end{array}
\right.
$$

Define $c_k$ by 
$$
\label{ck defn}
\dfrac{1}{c_k}  
= 
\displaystyle \P_k N(J_k) + \sum_{I \in Q(M_k)}  N(I). 
$$

By \eqref{QMsize} and assuming $M_k$ is large enough, we have 
\begin{align}\label{ck size}
c_k \sim \log(M_k) M_k^{-2d}. 
\end{align}
Recall that  $\tau > 1$. \label{etak defn} Define $\eta_k  = M_k^{-(1+\tau)}$. 
Define $F_k$ by 
\begin{align}\label{Fk defn}
F_k(x) 
= c_k 
\rbr{ 
\P_k \Phi_{J_k,\eta_k}(x) + \sum_{I \in Q(M_k)} \Phi_{I,\eta_k}(x)  
}. 
\end{align}

Notice that the definition of $F_k$ does not include $\Phi_{J_k,\eta_k}$ 
when $k$ is even. 
(In fact, we will never use $J_k$ for even $k$, and could define it arbitrarily.) 
This is to ensure that the support of the measure $\mu$ constructed from the sequence $F_k$  will be contained in $E(K,B,\tau)$. However, $\Phi_{J_k,\eta_k}$ is included for odd $k$ to obtain the failure of the restriction estimate (i.e., (iv) of Theorem \ref{mainthm3}). See Sections \ref{support of mu section} and \ref{sec_Failure_Estimate} for details. There is nothing special about the odds and evens per se; they could be replaced in the proof by any two disjoint infinite subsets of the positive integers. 

\begin{lemma}\label{Fk fourier lemma}
Define 
$$
\label{C0 defn}
C_0 = 2^{-1/d}C_B^{-1} \|T^{-1}\|^{-1} N(\delta)^{1/d}. 
$$
Let $k \geq 1$ and $s \in \ZZ^d$. 
\begin{align}
\label{F-1*}
\widehat{F_k}(0) &= 1.  
\\
\label{F-2*}
|\widehat{F_k}(s)| &\leq 1, \qquad \forall s \in \ZZ^d 
\\
\label{F-3*}
|\widehat{F_k}(s)| &=0, \qquad \text{if} \  0 < |s| \leq 
C_0 M_k^{1-\rho^{-}/d}.
\\
\label{F-4*}
|\widehat{F_k}(s)|
&\lesssim M_k^{-d+\rho^+} 
\log|s| \qquad \text{if} \ |s| 
\geq 
C_0 M_k^{1-\rho^{-}/d}.
\\
\label{F-5*}
|\widehat{F_k}(s)|
&\lesssim_N  
M_k^{-d+\rho^+}
\log |s| \left(\frac{M_k^{1+\tau}}{|s|} \right)^N \qquad \text{if} \ |s| \geq M_k^{1+\tau}. 
\end{align}
\end{lemma}

\begin{proof} 
\
By Lemma \ref{Phi-Fourier-Lemma}, we obtain 
$$
\widehat{F_k}(s) 
= c_k \widehat{\phi}(\eta_k s) \left( P_k N(J_k) \mathbf{1}_{\delta^{-1}J_k} (T^{-1} s)+ \sum_{I \in Q(M_k)} N(I) \one_{\delta^{-1} I}(T^{-1}s) \right).
$$

Since $\widehat{\phi}(0) = 1$, \eqref{F-1*} follows immediately. Since $F_k \geq 0$, \eqref{F-2*} follows from \eqref{F-1*}.

We prove the contrapositive of \eqref{F-3*}. 
Assume $s \neq 0$ and $\widehat{F_k}(s) \neq 0$. 
So there is at least one $I \in \cbr{J_{k}} \cup Q(M_k)$ such that $T^{-1} s \in \delta^{-1} I$. 
Then, by Lemma \ref{C1 bounds lemma} and Lemma \ref{idealquotient},
$$
N(\delta^{-1})N(I) \leq N(\abr{T^{-1} s}) = |N(T^{-1} s)| \leq C_B^d |T^{-1} s|^d \leq C_B^d \|T^{-1}\|^d |s|^d. 
$$
If $I \in Q(M_k)$, then $N(I) \geq M_k^d/2$. If $I = J_k$, then $N(I) \geq M_k^{d+\rho}/2$. 
Thus $$|s| \geq 2^{-1/d}C_B^{-1} \|T^{-1}\|^{-1} N(\delta)^{1/d} \min\cbr{M_k,M_k^{1+\rho/d}}.$$  

Now we prove \eqref{F-4*} and \eqref{F-5*}. 
Recall $N(J_k) \leq M_k^{d+\rho}$ and $N(I) \leq M_k^d$ for all $I \in Q(M_k)$. 
By \eqref{ck size}, $c_k \lesssim \log(M_k) M_k^{-2d}$. 
We need an upper bound on the number of $I \in Q(M_k)$ such that $T^{-1}s \in \delta^{-1} I$. 
Note that $T^{-1}s \in \delta^{-1} I$ implies $ (T^{-1}s) \delta \subseteq I$. So it suffices to find an upper bound on the number of $I \in Q(M_k)$ such that $ (T^{-1}s) \delta \subseteq  I$. By applying Lemma \ref{primedivisorbound} with $J = (T^{-1}s) \delta$ and then using \eqref{C1 norm bound} and \eqref{field-norm-ideal-norm-formula}, the number in question is seen to be $\lesssim \log |s|/\log M_k$, assuming $M_k$ (and hence $|s|$) is large enough.  
Therefore,
\[
|\widehat{F_k}(s)| 
\lesssim M_k^{-d} 
( M_k^{\rho^+}\log M_k+ \log |s|) |\widehat{\phi} (\eta_k s)|.
\]
Since $\phi$ is Schwartz, \eqref{F-4*} and \eqref{F-5*} follow immediately. 
\end{proof}

\section{Intermediate Measures}
\label{Intermediate Measures}

In this section, we define and establish Fourier-analytic estimates on measures $\mu_k$ and $\mu_{l,k}$.
These measures are used in the construction of $\mu$ and in proving the required properties of $\mu$.

\subsection{The measures \texorpdfstring{$\mu_k$}{muk}}\label{measure muk section}

Let $\chi:\RR^d \to \RR$ \label{chi defn} be a non-negative Schwartz function with support contained in  $([-3/8,-1/8] \cap [1/8,3/8])^d$ and $\widehat{\chi}(0) = \int \chi(x) dx = 1$.

We define measures $\mu_k$ \label{muk defn} on $\RR^d$ by
$$
d\mu_0 = \chi dx, \quad d\mu_k = \chi  F_{1} \cdots F_{k} dx \quad \text{ for all } k \in \NN.
$$
To simplify the statement and proof of the next lemma, we define $d \mu_{-1}= \chi dx$.

\begin{lemma}\label{mu_k_lemma}
Let $N$ be an integer with $N \geq 100 d$. \label{N defn}
If $(M_{k})_{k=0}^{\infty}$  grows sufficiently rapidly, then for all integers $k \geq 0$ and all $s \in \ZZ^d$, the following statements hold: 
\begin{equation}
\label{mu-1} |\widehat{\mu}_k(s) | \lesssim 1,   
\end{equation}

\begin{equation}
\label{mu-2} |\widehat{\mu}_k(s)  -\widehat{\mu}_{k-1}(s)| \lesssim 
M_k^{-(N-2d)}\log^2(M_k)  \qquad {\text{if}} \ |s| \leq 
\frac{1}{2}C_0 M_k^{1-\rho^{-}/d}, 
\end{equation}
\begin{equation}
\label{mu-3'}
|\widehat{\mu}_k(s)| \lesssim 
M_k^{-d+\rho^+} \log (M_k)
\log |s|  \qquad {\text{if}} \ |s| \geq 
\dfrac{1}{2} C_0 M_k^{1-\rho^{-}/d}, 
\end{equation}
\begin{equation}
\label{mu-4'}
|\widehat{\mu}_k(s)| \lesssim  
M_k^{-d+\rho^+}  \log (M_k)
\log |s| \left( \frac{M_k^{1+\tau}}{|s|} \right)^{N-d} \qquad {\text{if}} \ |s| \geq 2M_k^{1+\tau}.    
\end{equation}

\end{lemma}

\begin{proof}
We prove \eqref{mu-1}, \eqref{mu-2}, \eqref{mu-3'}, and \eqref{mu-4'} by induction. 

\textit{Base Step.}  
When $k=0$, \eqref{mu-1}, \eqref{mu-2}, \eqref{mu-3'}, \eqref{mu-4'} hold because $\chi$ is Schwartz and $\widehat{\chi}(0)=1$, assuming $M_0$ and the implicit constants are large enough.

\textit{Inductive Step.} 
Assume \eqref{mu-1}, \eqref{mu-2}, \eqref{mu-3'}, \eqref{mu-4'} hold with $k=0,\ldots,n$.
We shall prove \eqref{mu-1}, \eqref{mu-2}, \eqref{mu-3'}, \eqref{mu-4'} when $k=n+1$. We will prove \eqref{mu-1} last. 
For readability, define 
$D = M_{n+1}^{-d+\rho^+}$, 
$A = C_0 M_{n+1}^{1-\rho^{-}/d} $, 
$B = M_{n+1}^{1+\tau}$. 
Assume $M_{n+1}$ is large enough that 
\begin{align}\label{Dcond-standalone}
1 \leq D^{1/(d-N)}  \leq A \leq B
\end{align}
and 
\begin{align}\label{log(x)/x decreasing}
\text{$\log(x)/x$ is decreasing for all $x \geq A/2$.}
\end{align}
Also, we assume that $(M_k)_{k=0}^\infty$ grows rapidly so that
\begin{equation}\label{M_n_rapid}
    M_n^{-d+\rho^+ + (1+\tau)(N-d)} \log(M_n) \lesssim \log(M_{n+1})
\end{equation}
and that the implicit constants in  \eqref{mu-1}, \eqref{mu-2}, \eqref{mu-3'}, and \eqref{mu-4'} can be chosen uniformly in $k$.

\textit{Proof of \eqref{mu-2} when $k=n+1$.} 

Assume $|s| \leq  \A/2$. 
By \eqref{F-1*}, \eqref{F-2*}, and \eqref{F-3*},  
we have 
$$
|\widehat{F_{n+1}} \ast \widehat{\mu_n}(s) - \widehat{\mu_n}(s)| 
= \abs{\sum_{|t| \geq \A} \widehat{F_{n+1}}(t) \widehat{\mu_n}(s-t)}
\leq \sum_{|t| \geq \A} |\widehat{\mu_n}(s-t)|. 
$$
The assumption $|s| \leq  \A/2$ and the condition $|t| \geq \A$ on the sum imply that 
$|s-t| \geq \A/2$. 
Thus 
$$
|\widehat{F_{n+1}} \ast \widehat{\mu_n}(s) - \widehat{\mu_n}(s)| \leq \sum_{ |s-t| \geq \A/2 } |\widehat{\mu_n}(s-t)|. 
$$
Now our assumption of \eqref{mu-4'} when $k=n$ and \eqref{M_n_rapid} imply 
\begin{align*}
|\widehat{F_{n+1}} \ast \widehat{\mu_n}(s) - \widehat{\mu_n}(s)| 
&\lesssim 
\log(M_{n+1})\sum_{ |s-t| \geq \A/2  }|s-t|^{-(N-1-d)} \dfrac{\log(|s-t|)}{|s-t|}.  
\end{align*} 
By \eqref{log(x)/x decreasing} and because $N - 1 - d > d$, we conclude 
$$
|\widehat{F_{n+1}} \ast \widehat{\mu_n}(s) - \widehat{\mu_n}(s)| 
\lesssim 
\A^{-(N-2d)} \log(\A/2)\log(M_{n+1}) .  
$$

\textit{Proof of \eqref{mu-3'} when $k=n+1$.} 

Assume $|s| \geq \A/2$. By \eqref{F-1*} and \eqref{F-3*}, we have 
$$
\widehat{F_{n+1}} \ast \widehat{\mu_n}(s) = \widehat{\mu_n}(s) + S_1 + S_2 
$$
where 
\begin{align*}
S_1 &= \sum_{\substack{|s-t| \geq \A \\ |t| < |s|}} \widehat{F_{n+1}}(s-t) \widehat{\mu_n}(t), \\
S_2 &= \sum_{\substack{|s-t| \geq \A \\ |t| \geq |s|}} \widehat{F_{n+1}}(s-t) \widehat{\mu_n}(t).
\end{align*}
By our assumption of \eqref{mu-4'} when $k=n$ and \eqref{M_n_rapid}, 
$$
|\widehat{\mu_n}(s)| \lesssim \log(M_{n+1}) |s|^{-(N-d)} \log(|s|) \lesssim   \log(M_{n+1})
\A^{-(N-d)} \log(2|s|). 
$$
By \eqref{Dcond-standalone}, 
we have 
$$|\widehat{\mu_n}(s)| \lesssim   \log(M_{n+1}) \D\log(2|s|).$$
By \eqref{F-4*}, 
$$
|S_1| \lesssim 
\D \log(2|s|) \sum_{\substack{|s-t| \geq \A \\ |t| \leq |s|}} |\widehat{\mu_n}(t)|
\leq 
\D \log(2|s|) \sum_{t \in \ZZ^d} |\widehat{\mu_n}(t)|.  
$$
Because of our assumption of \eqref{mu-1}, \eqref{mu-4'} when $k=n$ and \eqref{M_n_rapid}, 
the last sum is $\lesssim \log(M_{n+1})$. 
Thus 
$$
|S_1| \lesssim  \log(M_{n+1})\D \log(2|s|). 
$$
On the other hand, by \eqref{F-4*} and our assumption of \eqref{mu-4'} when $k=n$ and \eqref{M_n_rapid}, 
$$
|S_2| \lesssim 
\log(M_{n+1}) \D \sum_{|t| \geq |s|} |t|^{-(N-d-2)} \dfrac{\log(2|t|)}{2|t|} \dfrac{\log(|t|)}{|t|}. 
$$
By \eqref{log(x)/x decreasing} 
and because $N-d-2>d$, 
$$
|S_2| \lesssim 
\log(M_{n+1}) \D |s|^{-(N-2d)} \log(2|s|) \log(|s|) 
\lesssim 
\log(M_{n+1}) \D\log(2|s|).   
$$
Combining all the estimates completes the proof of \eqref{mu-3'} when $k=n+1$.

\textit{Proof of \eqref{mu-4'} when $k=n+1$.} 

Assume $|s| \geq 2\B$. By \eqref{F-1*} and \eqref{F-3*}, we have 
$$
\widehat{F_{n+1}} \ast \widehat{\mu_n}(s) = \widehat{\mu_n}(s) + S_1 + S_2 + S_3,
$$
where 
\begin{align*}
S_1 &= \sum_{\substack{\A \leq |t| < \B \\ |s - t| \geq |s|/2}} \widehat{F_{n+1}}(t) \widehat{\mu_n}(s-t),  
\\
S_2 &= \sum_{\substack{ |s - t| < |s|/2}} \widehat{F_{n+1}}(t) \widehat{\mu_n}(s-t), \\
S_3 &= \sum_{\substack{|t| \geq \B \\ |s - t| \geq |s|/2}} \widehat{F_{n+1}}(t) \widehat{\mu_n}(s-t). 
\end{align*}
By our assumption of \eqref{mu-4'} when $k=n$ and \eqref{M_n_rapid}, 
$$
|\widehat{\mu_n}(s)| \lesssim \log(M_{n+1}) |s|^{-(N-d)} \log(|s|/2). 
$$
By \eqref{Dcond-standalone}, 
we have 
$$
|\widehat{\mu_n}(s)| \lesssim \log(M_{n+1}) \D (\B/|s|)^{N-d} \log(|s|/2). 
$$
By \eqref{F-4*} and our assumption of \eqref{mu-4'} when $k=n$ and \eqref{M_n_rapid}, 
\begin{align*}
|S_1| 
&\lesssim \log(M_{n+1}) |s|^{-(N-d)}\log(|s|/2) \sum_{\substack{\A \leq |t| < \B \\ |s - t| \geq |s|/2}} |\widehat{F_{n+1}}(t)|
\\
&\lesssim \log(M_{n+1}) |s|^{-(N-d)}\log(|s|/2) 
\sum_{\substack{\A \leq |t| < \B \\ |s - t| \geq |s|/2}} \D\log(|t|) 
\\
&\lesssim 
\log(M_{n+1}) \D |s|^{-(N-d)} \log(|s|/2) \log(\B) (\B)^d. 
\end{align*}
By \eqref{log(x)/x decreasing} and because $N-d>d+1$, 
we have 
$\log(\B) (\B)^d 
\lesssim \B^{N-d}.$
Thus 
$$
|S_1| \lesssim \log(M_{n+1})  \D (\B/|s|)^{N-d} \log(|s|/2).
$$
Since $|s-t| \leq |s|/2$ implies that $|t| \geq |s|/2 $, we use \eqref{F-5*} and obtain that  
\begin{align*}
|S_2| &\lesssim \D(\B/|s|)^{N-d} \log(|s|/2)
\sum_{\substack{ |s - t| \leq |s|/2}}  |\widehat{\mu_n}(s-t)| 
\\
&\leq 
\D(\B/|s|)^{N-d} \log(|s|/2) \sum_{t \in \ZZ^d}  |\widehat{\mu_n}(s-t)|.
\end{align*}
By \eqref{M_n_rapid}, 
and because of our assumption of \eqref{mu-1} and \eqref{mu-4'} when $k=n$, the last sum is bounded by a constant multiple of $\log(M_{n+1})$. 
Thus 
$$
|S_2| \lesssim \log(M_{n+1}) \D (\B/|s|)^{N-d} \log(|s|/2). 
$$
By \eqref{F-5*} and our assumption of \eqref{mu-4'} when $k=n$ and \eqref{M_n_rapid}, 
\begin{align*} 
|S_3| 
&\lesssim 
\log(M_{n+1}) |s|^{-(N - d)} \log(|s|/2) \sum_{\substack{|t| \geq \B \\ |s - t| \geq |s|/2}} |\widehat{F_{n+1}}(t)| \\
&\lesssim 
\log(M_{n+1}) \D (\B)^N |s|^{-(N - d)} \log(|s|/2) \sum_{\substack{|t| \geq \B \\ |s - t| \geq |s|/2}}  |t|^{-(N-1)} \dfrac{\log(|t|)}{|t|}. 
\end{align*}
Since $N-1>d$ and 
by \eqref{log(x)/x decreasing}, 
the last sum is $\lesssim (\B)^{-(N-d)}\log(\B)$. 
Therefore 
$$
|S_3| 
\lesssim 
\log(M_{n+1}) \D |s|^{-(N - d)} \log(|s|/2) \log(\B) (\B)^d. 
$$
Since $N-d>d+1$ and 
by \eqref{log(x)/x decreasing},  
we have 
$\log(\B) (\B)^d 
\lesssim (\B)^{N-d}. 
$
Thus 
$$|S_3| \lesssim \log(M_{n+1}) \D (\B/|s|)^{N-d} \log(|s|/2).$$
Combining all the estimates completes the proof of \eqref{mu-4'} when $k=n+1$.

\textit{Proof of \eqref{mu-1} when $k=n+1$.}

Recall we have assumed 
\eqref{mu-2} holds with $k=0,\ldots,n$, and we have proved \eqref{mu-2} when $k=n+1$. 
Therefore 
\begin{align*}
|\widehat{\mu_{n+1}}(0) - \widehat{\mu_0}(0)|
= 
\sum_{k=1}^{n+1} (\widehat{\mu_{k}}(0) - \widehat{\mu_{k-1}}(0))
\lesssim  \sum_{k=1}^{n+1} 
M_{k}^{-(N-2d)}\log^2(M_{k}).
\end{align*}
However,  
$\widehat{\mu_0}(0) = \widehat{\chi}(0) = 1$.
Furthermore, 
since $\mu_{n+1}$ has non-negative density, 
$|\widehat{\mu_{n+1}}(s)| \leq \widehat{\mu_{n+1}}(0)$. 
Thus, 
if $M_{1},\ldots,M_{n+1}$ are large enough, we get \eqref{mu-1} when $k=n+1$. 
\end{proof}

\subsection{The measures \texorpdfstring{$\mu_{l,k}$}{mul,k}}
\label{measures mulk section}

The following measures will be used to prove regularity and the failure of the restriction estimate for $\mu$.   

For $ l \geq k \geq 1$, define $$\mu_{l,k} \label{mulk defn}
:=  
F_{l} \cdots F_{{k+1}} F_{{k-1}} \cdots F_{1} \chi
= F_{l} \cdots F_{{k+1}} \mu_{k-1}. 
$$
Note $\mu_{k,k} := \mu_{k-1}$. 

We can prove an analog of Lemma \ref{mu_k_lemma} for $\mu_{l,k}$ using a virtually identical argument. However, we only need the following weaker version. We omit the proof for brevity.

\begin{lemma}\label{mu_k_ell_lemma}
For all $l > k \geq 1$ and $s \in \ZZ^d$, 
\begin{equation}\label{mukleq}
\left| \widehat{\mu_{l,k}} (s)  \right| \lesssim 1,  
\end{equation}
\begin{equation}\label{mukleq2}
\left| \widehat{\mu_{l,k}} (s) - \widehat{\mu_{l-1,k}}(s) \right| \leq \log^2(M_l)M_l^{-100} \qquad \text{if}\  |s| \leq  
\frac{1}{2}C_0 M_{l}^{1-\rho^{-}/d}.
\end{equation}
\end{lemma}

The following estimate will be used several times. 

\begin{lemma}\label{H mu convolution lemma}
Let $l > k \geq 1$.  
Let $C_k > 0$ and $H_k:\ZZ^d \to \RR$ such that, for all $u \in \ZZ^d$, 
\begin{align}\label{H mu convolution 1}
|H_k(u)| \leq C_k (1+\eta_k|u|)^{-10d}. 
\end{align}
Then, for all $s \in \ZZ^d$ such that $|s| \leq M_k^{1+\tau}$, 
\begin{align*}
\sum_{t \in \ZZ^d} |H_k(s-t)| 
| \widehat{\mu_{l,k}} (t) - \widehat{\mu_{k-1}}(t) |
\lesssim C_k M_{k}^{-80(1+\tau)d}. 
\end{align*}
\end{lemma}
\begin{proof}
Since $|s| \leq M_{k}^{1+\tau}$, 
if $|t| \geq M_k^{10(1+\tau)}$, then 
$|s-t| \geq |t|/2$, and so \eqref{H mu convolution 1} implies 
$$
|H_k(s-t)| \lesssim C_k (1+\eta_k |t|)^{-10d}. 
$$
We combine this with \eqref{mu-1} and \eqref{mukleq} to obtain 
\begin{align*}
&\sum_{|t| \geq M_k^{10(1+\tau)}} 
|H_k(s-t)| |\widehat{\mu_{l,k}} (t) - \widehat{\mu}_{k-1}(t)| 
\\
&
\lesssim \sum_{|t| \geq M_k^{10(1+\tau)}} C_k  \left( \frac{M_k^{1+\tau}}{|t|} \right)^{10d} 
\lesssim C_k M_k^{-80(1+\tau)d}.
\end{align*}
If $|t| \leq M_k^{10(1+\tau)}$,
since $M_{k}^{10(1+\tau)} \leq 
\frac{1}{2}C_0 M_{k+1}^{1-\rho^{-}/d}$, 
\eqref{mukleq2} implies that
$$
|\widehat{\mu_{l,k}} (t) - \widehat{\mu}_{k-1}(t)| \leq \sum_{j=k+1}^l \log^2(M_j) M_j^{-100} \lesssim M_{k+1}^{-99}.
$$
Then, since \eqref{H mu convolution 1} implies $|H_k| \lesssim C_k$, 
we obtain
\begin{align*}
&\sum_{|t| \leq M_k^{10(1+\tau)}}  |H_k (s-t)| |\widehat{\mu_{l,k}} (t) - \widehat{\mu}_{k-1}(t)| \\
&\lesssim C_k M_k^{10(1+\tau) d} M_{k+1}^{-99} \lesssim C_k M_k^{-80(1+\tau)d}.
\end{align*}
\end{proof}

\section{Construction and Fourier Analysis of  \texorpdfstring{$\mu$}{mu}}\label{sec_measure}

In this section we construct the measure $\mu$ and prove it satisfies 
the desired support and Fourier decay conditions (namely, conditions (i) and (iii) of Theorem \ref{mainthm3}). 

\subsection{Construction of  \texorpdfstring{$\mu$}{mu}}

\begin{lemma}\label{weakconv}
The sequence $(\mu_k)$ converges weakly to a non-zero finite Borel measure $\mu$. 
\end{lemma}

The proof of Lemma \ref{weakconv} requires the Lévy continuity theorem. \begin{thm}[Lévy Continuity Theorem]\label{Levy}
Let $(\nu_k)$ be a sequence of Borel measures on $\RR^d$ such that $0 < c \leq \nu_k(\RR^d) \leq C < \infty$ for all $k$. 
Suppose $(\widehat{\nu_k}(\xi))$ converges for each $\xi \in \RR^d$ 
and 
suppose $(\nu_k)$ is tight (i.e., for every $\epsilon>0$ there is a compact set $K_{\epsilon} \subseteq \RR^d$ such that $\nu_k(\RR^d \setminus K_{\epsilon}) < \epsilon$ for all $k$). Then there exists a finite non-zero Borel measure $\nu$ on $\RR^d$ such that $\nu_k$ converges to $\nu$ weakly (i.e., in distribution). 
\end{thm}

For a proof of this theorem when $\nu_k$ are probability measures, see  \cite[Section 26]{billingsley} for $d=1$ and \cite[Section 29]{billingsley} for $d \geq 1$. The general case follows by considering $a_k\nu_k$,  where $1/a_k = \nu_k(\RR^d)$. 
Alternatively, for a direct proof of the general case, see \cite[Section 8.8]{bogachev}. 

We also need the following lemma. 

\begin{lemma}\label{convseries}
Let $(\nu_k)$ be a sequence of Borel measures with support contained in $[-\frac{1}{2},\frac{1}{2}]^d$ and with $\sup_k \nu_k(\RR^d) < \infty$. If $(\widehat{\nu_k}(s))$ converges for each $s \in \mathbb{Z}^d$, then $(\widehat{\nu_k}(\xi))$ converges for each $\xi \in \mathbb{R}^d$. 
\end{lemma}
\begin{proof}
Let $\psi$ be a Schwartz function on $\RR^d$. 
For every $\xi \in \mathbb{R}^d$, we have (see \cite[Eq. (110)]{wolff-book}) 
\[\widehat{\psi \nu_k}(\xi) = \sum_{s \in \mathbb{Z}^d} \widehat{\nu_k}(s) \widehat{\psi}(\xi - s).\] 
Because $\widehat{\psi}$ is a Schwartz function and because $\sup_k |\widehat{\nu_k}(\xi)| \leq \sup_k \nu_k(\RR^d) < \infty$ for all $\xi \in \mathbb{R}^d$, the Lebesgue dominated convergence theorem lets us take the limit as $k \to \infty$. Thus, for each  $\xi \in \RR^d$, 
\[\lim_{k \to \infty} \widehat{\psi \nu_k}(\xi) = \lim_{k \to \infty} \sum_{s \in \mathbb{Z}^d} \widehat{\nu_k}(s) \widehat{\psi}(\xi - s) = \sum_{s \in \mathbb{Z}^d} \lim_{k \to \infty} \widehat{\nu_k}(s) \widehat{\psi}(\xi - s) \]
and the limit is finite. If $\psi$ is equal to 1 on $[-\frac{1}{2},\frac{1}{2}]^d$, then $\psi \nu_k = \nu_k$ for all $k$. 
\end{proof}
\begin{proof}[Proof of Lemma \ref{weakconv}] 
By construction, $\supp(\mu_k) \subseteq \supp(\chi) \subseteq [-\frac{1}{2},\frac{1}{2}]^d$ for all $k$. This also implies the sequence $(\mu_k)$ is tight. 
By \eqref{mu-2}, whenever $k \geq n$ and 
$|s| \leq C_0 M_k^{1-\rho^{-}/d}$,
we have 
\[\left|\widehat{\mu_{k}}(s) - \widehat{\mu_{n}}(s) \right| \leq \sum_{j=n}^{\infty} M_j^{-(N-2d)} \log^2(M_j).\]
In particular, $|\widehat{\mu_k}(0) - \widehat{\mu_0}(0)| \leq 1/2$. 
Since $\widehat{\mu_0}(0)= \widehat{\chi}(0) = 1$,  we have $1/2 \leq \mu_k(\RR^d) = \widehat{\mu_k}(0) \leq 3/2$ for all $k$. 
Furthermore, since the sum above goes to $0$ as $n \to \infty$, the sequence $(\widehat{\mu_k}(s))$ is Cauchy and hence convergent for each $s \in \ZZ^d$. 
Lemma \ref{convseries} implies $(\widehat{\mu_k}(\xi))$ converges for each $\xi \in \RR^d$. 
The Lévy continuity theorem now implies the desired result. 
\end{proof}

\subsection{Support of \texorpdfstring{$\mu$}{mu}}
\label{support of mu section}

\begin{prop}\label{support of mu lemma}
$\supp(\mu) \subseteq \supp(\chi) \cap E(K,B,\tau).$ 
\end{prop}
\begin{proof}
Because $\mu_k = \chi F_{1} \cdots F_{k}$ and because $\mu$ is the weak limit of $\mu_k$,  
\[\supp (\mu) \subseteq \bigcap_{k=1}^{\infty} \supp (\mu_k) =  \supp(\chi) \cap \bigcap_{k=1}^{\infty} \supp (F_{k}).\]
For even values of $k$, 
$$\supp (F_k) = \bigcup_{I \in Q(M_k)} \supp(\Phi_{I, \eta_k}).$$
If $I \in Q(M_k)$ and $\Phi_{I, \eta_k}(x) > 0$, then 
$\dist(x,I^{-1}) \leq \eta_k \leq N(I)^{-(\tau+1)/d}.$  
Since $I^{-1}$ is discrete, it follows that 
$$
\supp(\Phi_{I,\eta_k}) \subseteq \cbr{x \in \RR^d : \dist(x,I^{-1}) \leq N(I)^{-(\tau+1)/d}} 
$$
for every $I \in Q(M_k)$. 
By putting everything together and using that the sets $Q(M_k)$ are disjoint, we have 
$$
\bigcap_{k=1}^{\infty} \supp(F_k) \subseteq 
\bigcap_{\substack{k=1 \\ k \text{ even}}}^{\infty} \supp(F_k)
\subseteq 
E(K,B,\tau). 
$$
\end{proof}

\begin{remark}
If the definition of $F_k$ included the $\Phi_{J_k, \eta_k}$ term for all $k$ (instead of just for odd $k$), we would have 
$$
\supp (F_k) = \supp(\Phi_{J_k, \eta_k}) \cup \bigcup_{I \in Q(M_k)} \supp(\Phi_{I, \eta_k}) 
$$
for all $k$, and so $\supp(\mu)$ might not be a subset of $E(K,B,\tau)$. 
\end{remark}

\subsection{Fourier Decay of \texorpdfstring{$\mu$}{mu}}

\begin{prop}\label{Prop_mu_decay}
For all  $s \in \RR^d$ with $|s| \geq 2$, 
\begin{equation}\label{mu_decay}
|\widehat{\mu}(s)| \lesssim 
\log^2(|s|)|s|^{-\frac{d-\rho^+}{1 + \tau}}.
\end{equation}
\end{prop}

\begin{proof}
Since $\supp(\mu) \subseteq [-1/2,1/2]^d$, by a standard argument (see, e.g., \cite[Lemma 1, p.252-253]{kahane-book} or \cite[Lemma 9.A.4, p.69]{wolff-book}), we only need to prove \eqref{mu_decay} for $s \in \ZZ^d$. 
Moreover, it 
suffices to prove  that, for all integers $k \geq 0$, 
\begin{equation}\label{mu_k_decay}
|\widehat{\mu_k}(s) | \lesssim 
\log^2(|s|)|s|^{-\frac{d-\rho^+}{1 + \tau}} 
\quad \text{for all $s \in \ZZ, |s| \geq 2$}
, 
\end{equation}
where the implicit constant  does not depend on $k$.

We will show \eqref{mu_k_decay} by induction. For $k = 0$, the estimate \eqref{mu_k_decay} follows 
because $\mu_0 = \chi$ is a Schwartz function. Suppose $k \geq 1$ and suppose the statement \begin{equation}\label{mu_j_decay}
|\widehat{\mu_j}(s) | \lesssim 
\log^2(|s|)|s|^{-\frac{d-\rho^+}{1 + \tau}} \quad \text{for all $s \in \ZZ^d, |s| \geq 2$} 
\end{equation} holds for each $0 \leq j \leq k-1$ 
with implicit  constant independent of $j$. We consider three cases. 

Case 1: $2 \leq |s| \leq 
\frac{1}{2}C_0 M_k^{1-\rho^{-}/d}$. 
Let $0 \leq j_0 \leq k-1$ be such that 
$
\frac{1}{2}C_0 M_{j_0}^{1-\rho^{-}/d} 
\leq |s| \leq 
\frac{1}{2}C_0 M_{j_0+1}^{1-\rho^{-}/d}$,
taking $j_0 = 0$ if 
$|s| \leq 
\frac{1}{2}C_0 M_{0}^{1-\rho^{-}/d}$. 
By applying the triangle inequality, inequality \eqref{mu-2}, and inequality \eqref{mu_j_decay} with $j = j_0$, we have
\begin{align*}
|\widehat{\mu_k}(s)| & \leq |\widehat{\mu_{j_0}}(s)| + \sum_{j = j_0 + 1}^k |\widehat{\mu_j}(s) - \widehat{\mu_{j-1}}(s)| \\
& \lesssim 
\log^2(|s|)
|s|^{-\frac{d-\rho^+}{1 + \tau}}
+ \sum_{j=j_0 + 1}^k 
M_j^{-(N - 2d)}\log^2(M_j)
\\
& \lesssim 
\log^2(|s|)
|s|^{-\frac{d-\rho^+}{1 + \tau}}. 
\end{align*}

Case 2:  
$\frac{1}{2}C_0 M_{k}^{1-\rho^{-}/d} 
\leq |s| \leq 2M_k^{1+\tau}$.  Here \eqref{mu-3'} immediately implies \eqref{mu_k_decay}. 

Case 3: $|s| \geq 2M_k^{1+\tau}$. By our choice of $N$, we have 
$N-d \geq (d-\rho^+)/(1+\tau)$. Then \eqref{mu-4'} gives 
\begin{align*}
|\widehat{\mu}_k(s)| 
&\lesssim 
\log |s| 
\log(M_k) 
M_k^{-d+\rho^+}\left( \frac{M_k^{1+\tau}}{|s|} \right)^{N-d}
\\
&\leq 
\log |s| 
\log(M_k)
M_k^{-d+\rho^+}\left( \frac{M_k^{1+\tau}}{|s|} \right)^{\frac{d-\rho^+}{1+\tau}}
\\
&=
\log(|s|)
\log(M_k)
|s|^{-\frac{d-\rho^+}{1+\tau}}.
\end{align*}
This implies \eqref{mu_k_decay} for $|s| \geq 2M_k^{1+\tau}$. 
\end{proof}

\section{Regularity of \texorpdfstring{$\mu$}{mu}}\label{regularity_section}

In this section, we prove that $\mu$ satisfies (ii) of Theorem \ref{mainthm3}. We need several preparatory lemmas before getting to the main result, Proposition \ref{Prop_mu_reg}. 

\begin{lemma}\label{separation facts}
    Let $k \in \NN$.
\begin{enumerate}[(a)]
    \item Suppose $r_1, r_2 \in J_k^{-1}$. If $r_1 \neq r_2$,  then $|r_1 - r_2| \gtrsim M_{k}^{-1-\frac{\rho}{d}}$. 
    \item 
    Suppose $I_1,I_2 \in Q(M_k)$, $r_1 \in I_1^{-1}$, $r_2 \in I_2^{-1}$. 
    If $r_1 \neq r_2$, then $|r_1 - r_2| \gtrsim M_{k}^{-2}$. 
     \item Suppose $I_1,I_2 \in Q(M_k)$, 
     $r_1 \in I_1^{-1}$, $r_2 \in I_2^{-1}$. 
     If $r_1 = r_2$ and $I_1 \neq I_2$, then 
     $r_1=r_2 \in \OK \isom \ZZ^d$ and (consequently) 
     $\dist(r_1,\supp(\chi))=
     \dist(r_2,\supp(\chi)) > 3\eta_k$. 
\end{enumerate}
\end{lemma}
\begin{proof}
Lemma \ref{separation lemma} directly implies (a) and (b). 
Now we prove (c). 
Since $I_1$ and $I_2$ are distinct prime ideals of $\OK$, they are relatively prime, i.e., $1 \in I_1+I_2$.
Set $r = r_1=r_2$. 
Since $r \in I_1^{-1}$ and $r \in I_2^{-1}$, we have $rI_1 \subseteq \OK$ and $rI_2 \subseteq \OK$. 
Thus 
$
r = r \cdot 1 \in r(I_1+I_2) = rI_1 + rI_2 \subseteq 
\OK \isom \ZZ^d. 
$
Recalling that $\supp(\chi) \subseteq ([-3/8,-1/8] \cap [1/8,3/8])^d$ gives the result. 
\end{proof}

\begin{lemma}\label{Fk nonzero terms lemma}
Let $k \in \NN$. Let  $G_k:\RR^d \to \RR$. 
Suppose $\supp(G_k)$ has diameter $\leq 2\eta_k$ and 
suppose $\supp(G_k)$ intersects $\supp(\chi)$. 
Then there exists $I \in Q(M_k)$, $r' \in I^{-1}$, and $r'' \in J_k^{-1}$ 
such that, for all $x \in \RR^d$, 
\begin{align*}
F_k(x) G_k(x) = c_k \eta_k^{-d} 
\rbr{ \phi\rbr{\dfrac{x-r'}{\eta_k}} + P_k \phi\rbr{\dfrac{x-r''}{\eta_k}}}
G_k(x). 
\end{align*}
\end{lemma}
\begin{proof}
For brevity, write $\phi_{r,\eta_k}(x)=\phi((x-r)/\eta_k)$. 
By \eqref{Phi defn} and \eqref{Fk defn}, 
\begin{align}\label{Fk expanded 3}
F_k G_k = 
\sum_{r \in J_k^{-1}} P_k c_k  \eta_k^{-d} \phi_{r,\eta_k} G_k + 
\sum_{I \in Q(M)} \sum_{r \in I^{-1}} c_k  \eta_k^{-d} \phi_{r,\eta_k} G_k. 
\end{align}
We assume $M_k$ has been chosen so that  $\eta_k=M_k^{-(1+\tau)}$ is much smaller than $M_{k}^{-2}$
and $M_k^{-1-\rho/d}$. 
Note $\supp(\phi_{r,\eta_k}) \subseteq B(r,\eta_k)$. 
By Lemma \ref{separation facts}(a), $\supp(G_k)$ intersects $\supp(\phi_{r,\eta_k})$ for at most one $r \in J_k^{-1}$. 
Thus  
the first sum in \eqref{Fk expanded 3} has at most one non-zero term. 
Now consider two distinct pairs $(I_1,r_1)$ and $(I_2,r_2)$, where 
$I_1,I_2 \in Q(M_k)$, $r_1 \in I_1^{-1}$, $r_2 \in I_2^{-1}$. 
If $r_1 \neq r_2$, Lemma \ref{separation facts}(b) implies 
$\supp(G_k)$ intersects at most one of $\supp(\phi_{r_1,\eta_k})$ and $\supp(\phi_{r_2,\eta_k})$. 
If $r_1 = r_2$ and $I_1 \neq I_2$, Lemma \ref{separation facts}(c) implies 
$$
\dist(\supp(\phi_{r_1,\eta_k}),\supp(\chi))
= 
\dist(\supp(\phi_{r_2,\eta_k}),\supp(\chi))
> 
\eta_k,$$
so $\supp(G_k)$ intersects neither of $\supp(\phi_{r_1,\eta_k})$ and $\supp(\phi_{r_2,\eta_k})$. 
We conclude that the second sum in \eqref{Fk expanded 3} has at most one non-zero term. 
\end{proof}

\begin{lemma}\label{Fphi_decay}
Let $\psi$ be a Schwartz function with compact support. 
Let $x_0 \in \RR^d$ and $k \in \NN$. Let $B = B(x_0,\eta_k)$. Let $\psi_{B}(x) = \psi((x-x_0)/\eta_k)$. 
Suppose $\supp(\psi_B)$ intersects $\supp(\chi)$.
For all $Z \geq 0$ and $s \in \ZZ^d$, 
$$
\widehat{F_k \psi_{B}}(s)
\lesssim_{Z,\phi, \psi}c_k (1+P_k)(1+\eta_k|s|)^{-Z}.
$$
We emphasize that the implicit constant does not depend on $x_0$ or $k$. 
\end{lemma}
\begin{proof}
In light of Lemma \ref{Fk nonzero terms lemma}, it suffices to estimate 
$\widehat{fg}(s)$, where $f(x) = \phi\rbr{({x-r})/{\eta_k}}$ and $g(x) = \psi\rbr{({x-x_0})/{\eta_k}}$. 
Note
$$
|\widehat{fg}(s)| 
= 
|\widehat{f} \ast \widehat{g}(s)| 
\leq 
\int_{T_1} |\widehat{f}(t)| |\widehat{g}(s-t)| dt
+
\int_{T_2} |\widehat{f}(t)| |\widehat{g}(s-t)| dt, 
$$
where $T_1 = \cbr{t \in \RR^d: |s|/2 \leq |t|}$ and $T_2 = \cbr{t \in \RR^d: |s|/2 \leq |s-t|}$. 
Also observe that 
$|\widehat{f}(\xi)| = \eta_k^d |\widehat{\phi}(\eta_k \xi)|$ and 
$|\widehat{g}(\xi)| = \eta_k^d |\widehat{\psi}(\eta_k \xi)|$. 
To bound the integral over $T_2$, use that 
$|\widehat{g}(s-t)| \lesssim_{Z,\psi} \eta_k^d (1+\eta_k|s-t|)^{-Z} \leq \eta_k^d (1+\eta_k|s|/2)^{-Z}$ for all $t \in T_2$ and that $\|\widehat{f}\|_{L^1}=\|\widehat{\phi}\|_{L^1}$. The integral over $T_1$ is treated similarly. 
\end{proof}

\begin{lemma}\label{Fk bound lemma}
Let $k \in \NN$.  
\begin{enumerate}[(a)]
\item 
If $x_0 \in \supp(\chi)$, then 
\begin{align*}
F_k(x_0) \leq c_k \eta_k^{-d}(1+P_k) \|\phi\|_{\infty}. 
\end{align*}
\item 
If $x_0 \in \supp(\chi)$ and $x_0 \notin \supp(\Phi_{J_k,\eta_k})$, 
then 
$$
F_k(x_0) \leq c_k \eta_k^{-d} \|\phi\|_{\infty}. 
$$
\end{enumerate}
\end{lemma}
\begin{proof}
Take $G_k = \one_{\cbr{x_0}}$ in Lemma \ref{Fk nonzero terms lemma} and note that the support of $\Phi_{J_k,\eta_k}$ is the union of all closed balls $B(r,\eta_k)$ for $r \in J_k^{-1}$. 
\end{proof}

\begin{lemma}\label{k-ball regularity}
Let $k \in \NN$. Let $B$ be any ball of radius $\eta_k$. 
Let $2B$ be the ball with the same center but twice the radius. 
Then 
\begin{equation} \label{mu(B)-3}
\mu(B) \lesssim \log^2(M_k) 
M_k^{-2d+\rho^-}.
\end{equation}
Furthermore, if $2B$ does \textit{not} intersect $\supp(\Phi_{J_k,\eta_k})$, then 
\begin{equation} \label{mu(B)-1}
\mu(B) \lesssim \log^2(M_k) M_k^{-2d}
\end{equation}
and
\begin{equation}\label{mu(B)-2}
\mu(B) \lesssim \log^2(M_{k-1}) 
M_{k-1}^{d(\tau-1)+\rho^-}
|Q(M_k)|^{-1} M_k^{-d}.
\end{equation}
\end{lemma}
\begin{proof}
Note \eqref{mu(B)-1} follows from \eqref{mu(B)-2}, if $M_k$ is large enough. So we only need to prove \eqref{mu(B)-3} and \eqref{mu(B)-2}. 

\textit{Step 1: Estimate for $\mu_k(2B)$.} 
We first prove \eqref{mu(B)-3} and \eqref{mu(B)-2} with $\mu$ replaced by $\mu_k$ and $B$ replaced by $2B$. 
Note 
\begin{align*}
\mu_k(2B) = \int_{2B} \chi F_1 \cdots F_k dx. 
\end{align*}
Since the Lebesgue measure of $2B$ is $\lesssim M_k^{-(1+\tau)d}$, it will suffice to obtain appropriate pointwise bounds on the integrand.   
If $x \in \supp(\chi)$, then, by Lemma \ref{Fk bound lemma}(a), 
\begin{align*}
\chi F_1 \cdots F_k(x)
&\lesssim 
\|\chi\|_{\infty} 
\|\phi\|_{\infty}^k 
\prod_{i=1}^{k} 
c_i \eta_i^{-d}(1+P_i) \\
&=
\|\chi\|_{\infty} \|\phi\|_{\infty}^k
\rbr{\prod_{i=1}^{k-1}
c_i \eta_i^{-d}(1+P_i)} 
c_k \eta_k^{-d}(1+P_k) \\
&\lesssim \log(M_k) c_k \eta_k^{-d}(1+P_k) \\
\nonumber 
&\lesssim \log^2(M_k) 
M_k^{d(\tau-1)+\rho^-}, 
\end{align*}
assuming $M_k$ is large enough. 
This proves \eqref{mu(B)-3} with $\mu$ replaced by $\mu_k$ and $B$ replaced by $2B$.  
Now suppose $2B$ does not intersect $\supp(\Phi_{J_k,\eta_k})$. 
For each $x \in \supp(\chi) \cap 2B$, we apply Lemma \ref{Fk bound lemma}(a) to $F_1(x), \ldots, F_{k-1}(x)$ and apply Lemma \ref{Fk bound lemma}(b) to $F_k(x)$ to get  
\begin{align*}
\chi F_1 \cdots F_k(x) 
&\lesssim 
\|\chi\|_{\infty} \|\phi\|_{\infty}^k
\rbr{\prod_{i=1}^{k-1} 
c_i \eta_i^{-d}(1+P_i) } c_k \eta_k^{-d}  \\
&= \|\chi\|_{\infty} \|\phi\|_{\infty}^k  
\rbr{\prod_{i=1}^{k-2} 
c_i \eta_i^{-d}(1+P_i)} 
c_{k-1} \eta_{k-1}^{-d}(1+P_{k-1}) c_{k} \eta_{k}^{-d} \\
&\lesssim  \log(M_{k-1}) c_{k-1} \eta_{k-1}^{-d} (1+P_{k-1}) c_{k} \eta_{k}^{-d} \\
\nonumber 
&\lesssim \log^2(M_{k-1}) 
M_{k-1}^{d(\tau-1)+\rho^-}
|Q(M_k)|^{-1}M_k^{\tau d}, 
\end{align*}
assuming $M_{k-1}$ and $M_k$ are large enough. 
This proves \eqref{mu(B)-2} with $\mu$ replaced by $\mu_k$ and $B$ replaced by $2B$.

\textit{Step 2: Estimate for $\mu(B)$.} 
We assume $B$ intersects the support of $\chi$ (otherwise $\mu_l(B)=0$ and there is nothing to prove). 
Suppose $B = B(x_0,\eta_k)$. So $2B = B(x_0,2\eta_k)$. 
Let $\psi$ be a Schwartz function with support contained in $B(0,2)$, with $\psi \geq 0$ everywhere, and with $\psi \gtrsim 1$ on $B(0,1)$.
Let $\psi_{B}(x) = \psi((x-x_0)/\eta_k)$. 
For arbitrary $l \geq k$, we write
$$
\mu_l(\psi_B) : = \int \psi_B \mu_l dx.
$$
By Plancherel's theorem, we have
$$
\mu_l(\psi_B) - \mu_k (\psi_B)= \sum_{t \in \ZZ^d} \widehat{F_k\psi_B} (t) (\widehat{\mu_{l,k}} (t) - \widehat{\mu}_{k-1}(t)). 
$$
Note that hypotheses of Lemma \ref{Fphi_decay} are satisfied. 
By Lemma \ref{Fphi_decay} with $Z=10d$, the hypotheses of 
Lemma \ref{H mu convolution lemma} are satisfied with $C_k = c_k (1+P_k)$, $H(u) = \widehat{F_k \psi_B}(-u)$, and $s=0$. From Lemma \ref{H mu convolution lemma}, we obtain 
$$
\mu_l(\psi_B) - \mu_k(\psi_B) \lesssim c_k (1+P_k) M_k^{-80(1+\tau)d}.
$$
Since $\one_{B} \lesssim \psi_B \lesssim \one_{2B}$, we have $\mu_{l}(B) \lesssim \mu_{l}(\psi_B)$
and 
$\mu_k(\psi_B) \lesssim \mu_k(2B)$. 
Therefore 
$$
\mu_l(B) 
\lesssim \mu_k(2B) + c_k (1+P_k) M_k^{-80(1+\tau)d}.
$$
Letting $l \rightarrow \infty$, we have established 
\eqref{mu(B)-2} and \eqref{mu(B)-3}. 
\end{proof}

Now we prove the regularity of the measure $\mu$.

\begin{prop}\label{Prop_mu_reg}
For each ball $B_r$ in $\RR^d$ of radius $0 < r < 1$,
\begin{equation*}
\mu(B_r) \lesssim 
\log^2(r^{-1})r^{\frac{2d-\rho^-}{1+\tau}}.
\end{equation*}
\end{prop}

\begin{proof}
\textit{Preliminaries.} Since $\mu$ is a probability measure, the result holds for $r \geq M_0^{-(1+\tau)}$ by taking a large enough implicit  constant. 
Thus we assume that $M_k^{-(1+\tau)} \leq r \leq M_{k-1}^{-(1+\tau)}$ for some $k \geq 1$. 

We introduce some terminology. 
A $k$-ball is a closed ball of radius $\eta_k$ 
whose center is an element of $I^{-1}$ 
for some $I \in Q(M_k) \cup \cbr{J_k}$. 
We consider two kinds of $k$-balls. 
Let $\mathcal{B}_2$ be the set of $k$-balls $B$ such that $2B$ intersects $\supp(\Phi_{J_k,\eta_k})$ and let $\mathcal{B}_1$ be the remaining $k$-balls. (Recall that $2B$ is the ball with the same center as $B$ but with twice the radius.)

The support of $\mu$ is contained in the support of $F_k$, which is the union of all $k$-balls. 
Hence, we have the decomposition
$$
\mu(B_r) \leq \sum_{B \in \mathcal{B}_1, B \cap B_r \neq \emptyset} \mu(B) +\sum_{B \in \mathcal{B}_2, B \cap B_r \neq \emptyset} \mu(B).
$$

\textit{Sum over $B \in \mathcal{B}_1.$} 
First we estimate the sum over $B \in \mathcal{B}_1$. 
Since each ball in $\mathcal{B}_1$ is centered at an element of $I^{-1}$ for some $I \in Q(M_k)$, we will use Lemma \ref{separation lemma} to bound the number of terms in the sum. 
We consider three cases. 

Case 1: $M_k^{-(1+\tau)} \leq r \leq M_k^{-2} $. 
By Lemma \ref{separation lemma}, the centers of balls in $\mathcal{B}_1$ are separated by $\gtrsim M_k^{-2}$. 
Thus the number of terms in the sum over $\mathcal{B}_1$ is $\lesssim 1$. By \eqref{mu(B)-1}, 
$$
\sum_{B \in \mathcal{B}_1, B \cap B_r \neq \emptyset} \mu(B) 
\lesssim \log^2(M_k) M_k^{-2d} \lesssim \log^2(r^{-1}) r^{\frac{2d}{1+\tau}}.
$$

Case 2: $M_k^{-2} \leq r \leq M_k^{-1} $. Again, by Lemma \ref{separation lemma}, the centers of balls in $\mathcal{B}_1$ are separated by $\gtrsim M_k^{-2}$. Now the number of terms in the sum is $\lesssim (M_k^2r)^d$. By \eqref{mu(B)-1}, 
$$
\sum_{B \in \mathcal{B}_1, B \cap B_r \neq \emptyset} \mu(B) 
\lesssim \log^2(M_k) r^d \lesssim \log^2(r^{-1}) r^d \leq \log^2(r^{-1}) r^{\frac{2d}{1+\tau}}. 
$$

Case 3: $M_k^{-1} \leq r \leq M_{k-1}^{-(1+\tau)}$. Fix an ideal $I \in Q(M_k)$. By Lemma \ref{separation lemma}, the elements in $I^{-1}$ are at least $M_k^{-1}$ separated. 
So the number of $k$-balls in $\mathcal{B}_1$ that intersect $B_r$ and whose centers belong to $I^{-1}$ 
is $\lesssim (M_k r)^d$. 
Therefore the total number of $k$-balls in $\mathcal{B}_1$ that intersect $B_r$ is $\lesssim |Q(M_k)| (M_k r)^d$. 
Using  \eqref{mu(B)-2} and that $r \leq M_{k-1}^{-(1+\tau)}$, we get 
\begin{align*}
\sum_{B \in \mathcal{B}_1, B \cap B_r \neq \emptyset} \mu(B) \lesssim \log^2(M_{k-1}) 
M_{k-1}^{(\tau -1)d+\rho^-} 
r^d
 \lesssim \log^2 (r^{-1}) 
 r^{\frac{2d-\rho^-}{1+\tau}}. 
\end{align*}

\textit{Sum over $B \in \mathcal{B}_2.$} 
Now we consider the sum over $B \in \mathcal{B}_2$. 
Let $\mathcal{B}'_2$ be the subcollection of $\mathcal{B}_2$ consisting of the balls whose center is an element of $J_{k}^{-1}$. 
Let $\mathcal{B}''_2$ be the remaining balls in $\mathcal{B}_2$. 
By Lemma \ref{separation facts}(b), each ball in $\mathcal{B}'_2$ intersects at most one ball in $\mathcal{B}''_2$. 
Thus the cardinality of $\mathcal{B}_2$ is at most twice that of $\mathcal{B}'_2$. 
Moreover, by Lemma \ref{separation facts}(a), 
the balls in $\mathcal{B}'_2$ are separated by $\gtrsim M_k^{-1-\frac{\rho}{d}}$. 
Thus the cardinality of $\mathcal{B}_2$ is $ \lesssim \max\cbr{ 1, M_k^{d+\rho} r^d}$.
Then, by \eqref{mu(B)-3}, 
\begin{equation}\label{meas_J_2}
\sum_{B \in \mathcal{B}_2, B \cap B_r \neq \emptyset} \mu(B) \lesssim \log^2(M_k)
M_k^{-2d+\rho^-}
\cdot 
\max\cbr{1, M_k^{d+\rho} r^d}.
\end{equation}
If the maximum 
in \eqref{meas_J_2} 
is $1$, then $M_k^{-(1+\tau)} \leq r \leq M_k^{-1-\frac{\rho}{d}}$, and so 
$$
\sum_{B \in \mathcal{B}_2, B \cap B_r \neq \emptyset} \mu(B) \lesssim \log^2(r^{-1}) 
r^{\frac{2d+\rho^-}{1+\tau}}.
$$
If the maximum 
in \eqref{meas_J_2} 
is $M_k^{d +\rho} r^d$, 
then 
$$
\sum_{B \in \mathcal{B}_2, B \cap B_r \neq \emptyset} \mu(B) 
\lesssim 
\log^2(M_k) 
M_k^{-d+\rho^+}
r^d. 
$$
Since the function  
$\log^2(x)x^{-d+\rho^+}$ 
is eventually decreasing and 
since $M_k^{-(1+\tau)} \leq r \leq M_{k-1}^{-(1+\tau)}$, 
the last expression is 
\begin{align*}
\lesssim 
\log^2(r^{-1}) 
r^{\frac{d-\rho^+}{1+\tau}} 
r^d 
\leq 
\log^2(r^{-1})  
r^{\frac{2d-\rho^-}{1+\tau}}. 
\end{align*}
\end{proof}

\section{Failure of the Extension Estimate}\label{sec_Failure_Estimate}

In this section, we show that $\mu$ satisfies (iv) of Theorem \ref{mainthm3} with the functions $f_k = \Phi_{J_{2k+1}, \eta_{2k+1}}$. 
Explicitly, our goal is to prove the following proposition. 

\begin{prop}\label{res fail prop}
For all $1 \leq q < \infty$ and all $p < p(\tau, \rho, q, d) := \frac{q(d \tau - \rho)}{(q-1) (d-\rho^+)}$, 
$$
\lim_{\substack{k \to \infty \\ k \text{ odd}}} \dfrac{\|\widehat{\Phi_{J_k,\eta_k} \mu}\|_{L^p}}{\|\Phi_{J_k,\eta_k}\|_{L^q(\mu)}} = \infty.
$$
\end{prop}

Our proof of this proposition requires establishing several intermediate results.  

\subsection{Pointwise Lower Bound}\label{Pointwise Lower Bound}

Fix a constant $0 < c_{\phi} < 1$ \label{cphi defn} such that $\widehat{\phi^2}(\xi) \geq c_{\phi}$ for all $|\xi| \leq c_{\phi}$. (Such a constant exists because $\widehat{\phi^2}(0) = \int \phi^2(x) dx > 0$ and $\widehat{\phi^2}$ is continuous.)  
Fix an odd positive integer $k$. 
We will prove:  

\begin{lemma}\label{lem_lb_fkmu}
If $T^{-1} s \in \delta^{-1} J_k$ and $|s| \leq c_{\phi} M_k^{1 + \tau}$, then 
\begin{equation}\label{eq_lb_fkmu}
\left|\widehat{\Phi_{J_k, \eta_k} \mu}(s)\right| \gtrsim \log(M_k) 
M_k^{d \tau+\rho^+}.
\end{equation}
\end{lemma}

The rest of the subsection is devoted to the proof of this lemma. It suffices to show that 
\begin{equation}\label{f_kmu_l_lb}
|\widehat{\Phi_{J_k, \eta_k} \mu_l} (s)| \gtrsim \log(M_k) 
M_k^{d \tau+\rho^+} 
\end{equation}
for $l \geq k$. 
Fix $l \in \NN$ with $l \geq k$. 
Note that $\Phi_{J_k, \eta_k}F_k = h_1+h_2$, where
$$
h_1(x) = c_k P_k (\Phi_{J_k, \eta_k}(x))^2
$$
and
$$
h_2(x) = c_k \Phi_{J_k, \eta_k}(x) \sum_{I \in Q(M_k)} \Phi_{I, \eta_k}(x).
$$
Recall that 
$\mu_{l,k} = F_l F_{l-1} \cdots F_{k+1} \mu_{k-1}$. 
Therefore  $\Phi_{J_k, \eta_k} \mu_l  = h_1 \mu_{l,k} + h_2 \mu_{l,k}$, and hence 
\begin{align}\label{phi mu decomp}
\widehat{\Phi_{J_k, \eta_k} \mu_l}(s)
=
\widehat{h_1}(s) \widehat{\mu_{l,k}}(0)
+
\sum_{\substack{t \in \ZZ^d \\  t \neq 0}} \widehat{h_1}(s-t) \widehat{\mu_{l,k}}(t)
+
\sum_{t \in \ZZ^d} \widehat{h_2}(s-t) \widehat{\mu_{l,k}}(t). 
\end{align}

We will estimate the three terms on the right of \eqref{phi mu decomp} via a series of lemmas. Note that if $\rho = 0$, then the third term equals the sum of the first two, so only estimates involving $h_1$ are needed. 
Indeed, if $\rho = 0$, then $J_k \in Q(M_k)$, so parts (b) and (c) of Lemma \ref{separation facts} imply $\Phi_{J_k, \eta_k}(x)\Phi_{I, \eta_k} (x) = 0$ whenever $x \in \supp(\xi)$ and $I \neq J_k$ (assuming $M_k$ has been chosen so that  $\eta_k=M_k^{-(1+\tau)}$ is much smaller than $M_{k}^{-2}$), and therefore $h_2 = c_k(\Phi_{J_k,\eta_k})^2 = P_k^{-1} h_1 = h_1$ on the support of $\mu_{l,k}$. 

\begin{lemma}
If $T^{-1} s \in \delta^{-1} J$ and $|s| \leq c_{\phi} M_k^{1+\tau}$, then 
\begin{equation}\label{h_1_lb}
|\widehat{h_1} (s) \cdot  \widehat{\mu_{l,k}} (0) | \gtrsim \log(M_k) 
M_k^{d \tau+\rho^+}.
\end{equation}
\end{lemma}
\begin{proof}
By \eqref{Phi-Fourier-eqn-2},  
\eqref{ck size}, and the definition of $c_{\phi}$ (see the beginning of this section), we have 
\begin{align*}
|\widehat{h_1}(s)| 
= c_k P_k \eta_k^{-d}N(J_k) |\widehat{\phi^2}(\eta_k s)| 
\gtrsim  
\log(M_k) 
M_k^{d\tau+\rho^+}.
\end{align*}
Since $\widehat{\mu_{0}}(0) \widehat{\phi}_0(0)=1$, we iterate \eqref{mukleq2} and obtain  $|\widehat{{\mu}_{l,k}}(0)| \gtrsim 1$. 
\end{proof}

\begin{lemma}
If $T^{-1} s \in \delta^{-1} J$ and $|s| \leq M_k^{1+\tau}$, then 
\begin{equation}\label{h_1_ub}
\sum_{t \neq 0 }|\widehat{h_1} (s-t)   \widehat{\mu_{l,k}} (t) | \lesssim \log(M_k) 
M_k^{d \tau + \rho^+ - 80(d+\rho)}.
\end{equation}
\end{lemma}

\begin{proof}
By \eqref{Phi-Fourier-eqn-2} and \eqref{ck size}, 
\begin{align*}
\widehat{h_1}(u) & = c_k P_k \widehat{\Phi_{J_k, \eta_k}^2}(u) \\
& \lesssim c_k \eta_k^{-d} P_k \left|\widehat{\phi^2}(\eta_k u)\right| N(J_k) \mathbf{1}_{\delta^{-1} J_k}  (T^{-1} u) \\
& \lesssim 
\log (M_k) 
M_k^{d \tau+\rho^+} 
\left(1 + \frac{|u|}{M_k^{1 + \tau}} \right)^{-10d} \mathbf{1}_{\delta^{-1} J_k}(T^{-1} u).
\end{align*}
Therefore, the hypotheses of Lemma \ref{H mu convolution lemma} are satisfied with 
$C_k = \log (M_k) 
M_k^{d \tau+\rho^+}$  
and $H_k(u) = \widehat{h_1}(u) \one_{\ZZ^d \setminus \cbr{s}}(u)$.  
Consequently, using the assumption $|s| \leq M_k^{1+\tau}$, we get 
$$
\sum_{t \neq 0} |\widehat{h_1}(s-t) \widehat{\mu_{l,k}}(t)|
\lesssim C_k M_{k}^{-80(1+\tau)d} + \sum_{t \neq 0} |\widehat{h_1}(s-t) \widehat{\mu_{k-1}}(t)|. 
$$
Since $T^{-1}s \in \delta^{-1} J$ and since $\widehat{h}_1(s-t)$ is zero when $T^{-1}(s-t) \notin \delta^{-1} J$, we conclude that $T^{-1}t \in \delta^{-1} J$ for those $t$ such that $\widehat{h_1}(s - t) \neq 0$. 
Thus, by \eqref{C1 norm bound} and Lemma \ref{idealquotient}, 
$|t|^d \gtrsim N(J) \gtrsim M_k^{d+ \rho}$, 
which implies $|t| \geq 2M_{k-1}^{1+\tau}$. 
Thus, by \eqref{mu-4'}, we have
\begin{align*}
|\widehat{\mu_{k-1}}(t) | 
\lesssim 
\log(M_{k-1}) 
M_{k-1}^{-d+\rho^+} 
\log(|t|)  ( M_{k-1}^{1+\tau} / |t|)^{99d}. 
\end{align*}
Recall that $|\widehat{h}_1 (u)| \lesssim \log(M_k) 
M_k^{d \tau + \rho^+} 
= C_k$. 
Therefore 
$$
\sum_{t \neq 0} |\widehat{h_1}(s-t) \widehat{\mu_{k-1}}(t)|
= 
\sum_{|t| \gtrsim M_k^{1+\frac{\rho}{d}}} |\widehat{h_1}(s-t) \widehat{\mu_{k-1}}(t)|
\lesssim C_k M_k^{-97(d + \rho)}. 
$$
Since $80(1+\tau)d$ and $97(d+\rho)$ are both larger than $80(d+\rho)$, 
we have established \eqref{h_1_ub}. 
\end{proof}

To bound the sum involving $\widehat{h_2}$ in \eqref{phi mu decomp}, we first prove the following lemma. 
\begin{lemma}\label{h_2_est}
Suppose $\rho \neq 0$. 
For all 
$s \in \ZZ^d$, 
$$
|\widehat{h_2}(s)| \lesssim  M_k^{d\tau} (1+ M_k^{\rho -d(\tau-1)})(1 + \eta_k |s|)^{-10d}.
$$
\end{lemma}
\begin{proof}
Note that
\begin{align*}
\widehat{h_2}(s) 
= \sum_{I \in Q(M_k)} \sum_{t \in \mathbb{Z}^d} 
c_k \widehat{\Phi_{I, \eta_k}}(t) \widehat{\Phi_{J_k, \eta_k}}(s - t).
\end{align*}
By Lemma \ref{Phi-Fourier-Lemma}, 
\[\widehat{h_2}(s) = \sum_{I \in Q(M_k)} \sum_{t \in \mathcal{T}(I)} 
c_k 
N(I) 
 N(J_k) 
 \widehat \phi(\eta_k t)
\widehat \phi(\eta_k(s - t)),\]
where 
$$
\mathcal{T}(I)
=
\cbr{t \in \ZZ^d : T^{-1}t \in \delta^{-1}I \cap (T^{-1}s + \delta^{-1}J_k)}.
$$
Since at least one of $|s - t|$ and $|t|$ will be $\geq |s|/2$, 
the Schwartz bounds on $\widehat{\phi}$ 
imply 
$$
|\widehat \phi(\eta_k t)
\widehat \phi(\eta_k(s - t))| 
\lesssim (1+\eta_k|s|)^{-10d}(1+\eta_k|t|)^{-50d}.
$$
Therefore,
\begin{equation}\label{h2hatsum2}
|\widehat{h_2}(s)| \lesssim 
(1 + \eta_k |s|)^{-10d} \log (M_k) M_k^{\rho} 
\sum_{I \in Q(M_k)} 
\sum_{t \in \mathcal{T}(I)} 
(1+\eta_k|t|)^{-50d}. 
\end{equation}
Consider a fixed $I \in Q(M_k)$. 
By Lemma \ref{concrete crt lemma},  either the set 
$\delta^{-1} I \cap (T^{-1} s + \delta^{-1} J_k)$ is empty or 
there exists $a \in \delta^{-1}$ such that 
$$
\delta^{-1} I \cap (T^{-1} s + \delta^{-1} J_k)
=
a + \delta^{-1} I \cap \delta^{-1}J_k. 
$$ 
In the first case, $\mathcal{T}(I)$ is empty, so the sum over $t \in \mathcal{T}(I)$ in \eqref{h2hatsum2} is zero. 
In the second case, 
$\delta^{-1} I \cap (T^{-1} s + \delta^{-1} J_k)$ is a translate of 
$\delta^{-1} I \cap \delta^{-1}J_k$, 
so the separation between the elements of $\delta^{-1} I \cap (T^{-1} s + \delta^{-1} J_k)$
 is the same as the separation between the elements of $\delta^{-1} I \cap \delta^{-1}J_k$. 
By Lemma \ref{separation lemma} (with $I_1 = I_2 = ( \delta^{-1} I \cap \delta^{-1}J_k )^{-1}$), the elements of $\delta^{-1} I \cap \delta^{-1}J_k$ are separated by 
$
\gtrsim 
N(\delta^{-1} I \cap \delta^{-1}J_k)^{1/d}.
$ 
But, by Lemma \ref{idealquotient}, 
\begin{align}\label{I Jk norm calc}
N(\delta^{-1} I \cap \delta^{-1}J_k)^{1/d}
\gtrsim  
N(I \cap J_k)^{1/d}
=(N(I) N(J_k))^{1/d}
\gtrsim 
M_k^{2+\frac{\rho}{d}}, 
\end{align}
where the equality holds because $I$ and $J_k$ are relatively prime. 
Thus the elements of 
$\delta^{-1} I \cap (T^{-1} s + \delta^{-1} J)$ 
are separated by 
$\gtrsim 
M_k^{2+\frac{\rho}{d}}.$ 
Since $T$ is a bounded linear isomorphism, 
the elements of $\mathcal{T}(I)$ 
are also separated by $\gtrsim M_k^{2+\frac{\rho}{d}}$. 
It follows that 
\begin{align*}
\sum_{t \in \mathcal{T}(I)} (1+\eta_k|t|)^{-50d}
&\lesssim
\sum_{t \in \ZZ^d} (1+\eta_k M_k^{2+\frac{\rho}{d}}|t|)^{-50d} 
\\
&\lesssim 
\left\{\begin{array}{cc}
    (M_k^{1+\tau} / M_k^{2+\frac{\rho}{d}})^d  & \text{if } \rho \leq d(\tau-1),  \\
    1 & \text{if } \rho > d(\tau-1). 
\end{array}\right.
\end{align*}
In any case, the sum is $\lesssim M_k^{d(\tau-1) -\rho}+1$. 
Plugging this estimate and \eqref{QMsize} into \eqref{h2hatsum2} 
completes the proof. 
\end{proof}

Now we are ready to bound the sum involving $\widehat{h_2}$ in \eqref{phi mu decomp}.

\begin{lemma}
Suppose $\rho \neq 0$. If $T^{-1} s \in \delta^{-1} J$ and $|s| \leq M_k^{1+\tau}$, then 
\begin{equation}\label{h_2_ub}
\sum_{t \in \ZZ^d} |\widehat{h_2}(s-t) \widehat{\mu_{l,k}}(t)| \lesssim \log^{1/10}(M_k) M_k^{d\tau} (1+ M_k^{\rho -d(\tau-1)}).  
\end{equation}
\end{lemma}

\begin{proof}
By Lemma \ref{h_2_est}, the hypotheses of Lemma \ref{H mu convolution lemma} are satisfied with 
$C_k = M_k^{d \tau} (1 + M_k^{\rho-d(\tau-1)})$ and $H_k = \widehat{h_2}$.  
Consequently, using the assumption $|s| \leq M_k^{1+\tau}$, we get 
\begin{align*}
\sum_{t \in \ZZ^d }|\widehat{h_2} (s-t) \widehat{\mu_{l,k}} (t) | 
\lesssim 
C_k M_k^{-80d(1+\tau)} 
+ 
\sum_{t \in \ZZ^d }
|\widehat{h_2} (s-t) 
\widehat{\mu_{k-1}} (t)|.  
\end{align*}
To bound the sum on the right, we consider separately the ranges $|t| < \log^{1/10d}(M_k)$ and $|t| \geq  \log^{1/10d}(M_k)$.  
In both cases, we shall use the bound  $|\widehat{h}_2(u)| \lesssim  M_k^{d\tau} (1+ M_k^{\rho -d(\tau-1)})$ furnished by Lemma \ref{h_2_est}. 
By \eqref{mu-1}, we have
\begin{align*}
\sum_{|t| < \log^{1/10d}(M_k)} |\widehat{h}_2(s-t) \widehat{\mu}_{k-1}(t)| 
\lesssim  \log^{1/10}(M_k) M_k^{d\tau} (1+ M_k^{\rho -d(\tau-1)}). 
\end{align*}
We assume $M_k$ is large enough that $\log^{1/20d}(M_k) \geq 2M_{k-1}^{1+\tau}$. 
Thus we can apply \eqref{mu-4'} (with $k-1$ in place of $k$) to obtain 
\begin{align*}
\sum_{|t| \geq \log^{1/10d}(M_k)} |\widehat{h}_2(s-t) \widehat{\mu}_{k-1}(t)| 
\lesssim M_k^{d\tau} (1+ M_k^{\rho -d(\tau-1)}).
\end{align*}
\end{proof}

We now complete the proof of Lemma \ref{lem_lb_fkmu}. If $\rho \neq 0$, \eqref{f_kmu_l_lb} follows by plugging the estimates \eqref{h_1_lb}, \eqref{h_1_ub}, \eqref{h_2_ub} into 
\eqref{phi mu decomp}. 
As we observed near the start of the proof, if $\rho = 0$, then only the estimates involving $h_1$ are needed, and so \eqref{f_kmu_l_lb} follows from \eqref{h_1_lb} and \eqref{h_1_ub}.

\subsection{Norm  Bounds}\label{norm bounds}

The next lemma is an upper bound on the denominator in Proposition \ref{res fail prop}. It will also be combined with Lemma \ref{lem_lb_fkmu} above and Lemma \ref{count lemma} below to produce a lower bound on the numerator in Proposition \ref{res fail prop}. 

\begin{lemma}\label{PhiLqnorm}
For each $1 \leq q < \infty$, 
$$
\norm{\Phi_{J_k, \eta_k}}_{L^q(\mu)}^q
\lesssim 
\log^2(M_k) M_k^{qd(1 + \tau) - d + \rho^+}.
$$
\end{lemma}
\begin{proof}
Note $\Phi_{J_k, \eta_k}(x)$ is supported on balls of radius $\eta_k = M_k^{-(1+\tau)}$ which 
(by Lemma \ref{separation facts}) are separated by $\gtrsim M_k^{-1-\frac{\rho}{d}}$. 
As the support of $\mu$ is compact, it intersects $\lesssim M_{k}^{d+\rho}$ of these balls. 
So, by \eqref{mu(B)-3} of Lemma \ref{k-ball regularity}, 
$$
\mu\rbr{\supp\rbr{\Phi_{J_k, \eta_k}}}
\lesssim 
\log^2(M_k)M_{k}^{-d+\rho^+}. 
$$
Since $\Phi_{J_k, \eta_k}$ is bounded above by $\eta_k^{-d} = M_k^{d(1 + \tau)}$, 
the desired result is immediate. 
\end{proof}

To obtain a lower bound on the numerator in Proposition \ref{res fail prop}, we need the following counting lemma. 

\begin{lemma}\label{count lemma} 
The number of $s \in \ZZ^d$ such that $T^{-1} s \in \delta^{-1} J_k$ and $|s| \leq c_{\phi}M_k^{1+\tau}$ is 
$\gtrsim M_k^{d\tau - \rho}. $
\end{lemma}
\begin{proof}
Let $S(J_k)$ be the set of $s \in \ZZ^d$ such that $T^{-1} s \in \delta^{-1} J_k$ and $|s| \leq c_{\phi} M_k^{1+\tau}$. 
Consider $\cbr{y+J_k : y \in \ZZ^d, \, |y| \leq C_d N(J_k)^{1/d}} \subseteq  \OK/J_k$. 
Here, the constant $C_d$ is chosen so that the number of points $y \in \ZZ^d$ such that $|y| \leq C_d N(J_k)^{1/d}$ is strictly greater than $N(J_k) = |\OK/J_k|$. 
By the pigeonhole principle, there exist distinct $y_1,y_2 \in \ZZ^d$ such that $|y_1|,|y_2| \leq N(J_k)$ and $y_1+J_k = y_2+J_k$. 
Then $y_0 := y_1 - y_2$ is a non-zero element of $J_k$ such that $|y_0| \leq 2C_dN(J_k)^{1/d} \lesssim M_k^{1+\frac{\rho}{d}}$. 
Recall that $\ZZ^d \subseteq \delta^{-1}$ and $T(\delta^{-1}) = \ZZ^d$. 
Thus $zy_0 \in \delta^{-1} J_k$ and $T(zy_0) \in \ZZ^d$ for all $z \in \ZZ^d$. 
Moreover, for all $z \in \ZZ^d$, 
$|T(zy_0)|
\lesssim 
|zy_0|
\lesssim 
|z||y_0|
\lesssim |z|M_k^{1+\frac{\rho}{d}}$. 
So there exists a constant $C>0$ such that $|T(zy_0)| \leq c M_k^{1+\tau}$ whenever $z \in \ZZ^d$ and $|z| \leq CM_k^{\tau-\frac{\rho}{d}}$. 
Therefore $T(zy_0) \in S(J_k)$ for all  $z \in \ZZ^d$ with $|z| \leq  CM_k^{\tau - \frac{\rho}{d}}$. 
The number of $z \in \ZZ^d$ with $|z| \leq C M_k^{\tau - \frac{\rho}{d}}$ is  $\gtrsim M_k^{d\tau - \rho}$. 
Since $T$ is invertible and $y_0$ is non-zero, $T(z_1 y_0) \neq T(z_2 y_0)$ whenever $z_1,z_2$ are distinct elements of $\ZZ^d$. 
Therefore the number of elements of $S(J_k)$ is also $\gtrsim M_k^{d\tau - \rho}$. 
\end{proof}

We now prove a lower bound on the numerator in Proposition \ref{res fail prop}. 

\begin{lemma}\label{PhihatLpnorm v2}
Let $1 \leq p < \infty$. Then 
\[\norm{\widehat{\Phi_{J_k, \eta_k} \mu} }_{L^p}^p \gtrsim \log^{{p - 2d}}(M_k)  
M_k^{(d \tau + \rho^+) p}
M_k^{ (d \tau - \rho)}
\]
\end{lemma}
\begin{proof}
For all $s \in \ZZ^d$ and $\xi \in \RR^d$, 
\begin{align*}
\left| \widehat{\Phi_{J_k, \eta_k} \mu}(s + \xi) - \widehat{\Phi_{J_k, \eta_k} \mu}(s) \right| & = \left| \int \Phi_{J_k, \eta_k}(x) \left(e^{-2 \pi i (\xi + s) \cdot x} - e^{-2 \pi i s \cdot x} \right) \, d \mu(x) \right|  \\
& \leq \int \Phi_{J_k, \eta_k}(x) 2 \pi |x| |\xi| \, d \mu(x) \\
& \lesssim |\xi| \int \Phi_{J_k, \eta_k}(x) \, d \mu(x). \\
&\lesssim 
|\xi| \log^2(M_k) 
M_k^{d \tau+\rho^+}.
\end{align*}
For the second last inequality, we used the compactness of $\supp(\mu)$. 
For the last inequality, we applied Lemma \ref{PhiLqnorm} with $q = 1$. 
By Lemma \ref{lem_lb_fkmu} and the inequality we have just proved, 
if $|\xi| \leq {1}/{\log^2(M_k)}$ and 
$s \in \ZZ^d$ such that $T^{-1} s \in \delta^{-1} J_k$ and $|s| \leq c_{\phi} M_k^{1+\tau}$, then 
\[
\left|\widehat{\Phi_{J_k, \eta_k} \mu}(s + \xi)\right| \gtrsim 
\log(M_k) 
M_k^{d \tau + \rho^+}.
\]
Combining this with Lemma \ref{count lemma} implies the desired result. 
\end{proof}

\subsection{Completing the Proof of Proposition \ref{res fail prop}}
\label{ Completing the Proof of Proposition res fail prop}

By Lemma \ref{PhiLqnorm} and Lemma \ref{PhihatLpnorm v2}, for each odd $k$, 
\begin{align*}
\frac{\norm{\widehat{\Phi_{J_k, \eta_k} \mu}}_{L^p}}{\norm{\Phi_{J_k, \eta_k}}_{L^q(\mu)}} 
\gtrsim  
(\log(M_k))^{\frac{p - 2d}{p} - \frac{2}{q}}
M_k^{\frac{1}{p}(d\tau - \rho)  + \frac{1}{q}(q-1)(\rho^+ - d)}
\end{align*}
If $p < p(\tau, \rho, q, d) := \frac{q(d \tau - \rho)}{(q-1) (d-\rho^+)}$, the exponent of $M_k$ is positive.

\section{Hausdorff Dimension Upper Bound}\label{Hausdorff Dimension Upper Bound section}

\begin{prop}\label{hausdorff upper bound}
$\dim_H(E(K,B,\tau)) \leq 2d/(1+\tau)$.  
\end{prop}
\begin{proof}
Let $\mathcal{I}(\OK)$ denote the set of all ideals of $\OK$. 
Note $E(K,B,\tau)$ is invariant under translation by elements of $\ZZ^d$ because for $\ZZ^d \isom \OK \subseteq I^{-1}$ for every $I \in \mathcal{I}(\OK)$. 
Thus it suffices to show that $E(K,B,\tau) \cap [-1/2,1/2)^d$ has Hausdorff dimension at most $2d/(1+\tau)$. 
Let $B(r,M)$ denote the closed ball in $\RR^d$ with center $r$ and radius $M$. 
Note 
\begin{align*}
E(K,B,\tau) 
 = \bigcap_{n=1}^{\infty} \bigcup_{\substack{ I \in \mathcal{I}(\OK) \\ N(I) > n}}  \bigcup_{r \in I^{-1}} 
 B( r,N(I)^{-(1+\tau)/d} ). 
\end{align*}
So, for every $n \geq 1$,  
\begin{align*}
E(K,B,\tau) \cap [-1/2,1/2)^d 
\subseteq \bigcup_{\substack{ I \in \mathcal{I}(\OK) \\ N(I) > n}} 
\bigcup_{\substack{r \in I^{-1} \\ |r| < 2}} 
 B( r,N(I)^{-(1+\tau)/d} ). 
\end{align*}
Since $\OK \isom \ZZ^d$, for every $a \in \RR^d$, 
the set $I^{-1} \cap (a+[0,1)^d)$ is a complete set of representatives for $I^{-1}/\OK$. Thus the set of $r \in I^{-1}$ with $|r|<2$ is covered by at most $4^d$ complete sets of representatives of $I^{-1}/\OK$. By Lemma \ref{idealquotient}, 
$|I^{-1} / \OK| 
= {N(\OK)}/{N(I^{-1})} 
= N(I)$.  
Therefore the set of $r \in I^{-1}$ such that $|r| < 2$ has size $\lesssim N(I)$. Thus, for every $s \geq 0$, 
\begin{align*}
\sum_{\substack{ I \in \mathcal{I}(\OK) \\ N(I) > n}} 
\sum_{\substack{r \in I^{-1}\\ |r| < 2}} 
	\rbr{ \text{diam}( B( r,N(I)^{-(1+\tau)/d} ) ) }^s
\lesssim   
\sum_{\substack{ I \in \mathcal{I}(\OK) \\ N(I) > n}} N(I)^{1-s(1+\tau)/d}. 
\end{align*}
If $s > 2d/(1+\tau)$, the last sum goes to zero as $n \to \infty$. (This is because the sum $\zeta_{K}(z) := \sum_{ I \in \mathcal{I}(\OK) } N(I)^{-z}$, known as the Dedekind zeta function, converges whenever $\Re(z)>1$; see \cite[Ch.10]{Jarvis14} for a proof.)
Therefore the Hausdorff dimension of $E(K,B,\tau) \cap [-1/2,1/2]^d$ is at most $2d/(1+\tau)$. 
\end{proof}

\section{Appendix: modifications for \texorpdfstring{$E_{\prin}(K,B,\tau)$}{Eprin(K,B,tau)}}\label{sec_prin}
We will describe the modifications necessary to prove that Theorem \ref{mainthm3} holds for the set $E_{\prin}(K,B,\tau)$. The key 
is to modify the argument to apply to \textit{principal} ideals. This does not change the argument much because of the following standard fact from algebraic number theory.
\begin{thm}[Finiteness of the ideal class group]\label{idealclassthm}
Let $K$ be an algebraic number field. Then there exists a finite collection $\mathcal{H}$ of integral ideals such that if $I$ is an ideal of $\OK$, then there exists $H_I \in \mathcal{H}$ such that the ideal $I H_I$ is a principal ideal.
\end{thm}

Using Theorem \ref{idealclassthm}, we define the following analog of $Q(M)$ consisting of principal ideals: 
\[Q_{\prin}(M) = \left\{I H_I : I \in Q(M/R) \right\}\]
where 
$$
R = \max\cbr{ N((H_{1}^{-1} + H_{2}^{-1})^{-1}) : H_1, H_2 \in \mathcal{H} }. 
$$
Then $Q_{\prin}(M)$ is a collection of $\sim {M^d}/{\log M}$ principal ideals of norm $\sim M^d$ that are ``almost" prime ideals. 

Everywhere in the original argument, we replace $Q(M)$ by $Q_{\prin}(M)$.   
In particular, this means $J_k \in Q_{\prin}(M_k)$. 

Lemma \ref{primedivisorbound} is replaced by the following lemma (which is actually a corollary of Lemma \ref{primedivisorbound}). 

\begin{lemma}\label{principaldivisorbound}
Let $J$ be a nonzero integral ideal of $\OK$. Then 
\begin{equation}\label{principaldivisorineq}
|\{I \in Q_{\prin}(M) : J \subseteq I \}| \leq \frac{\log N(J)}{\log (M^d/2R^d)}.
\end{equation}
\end{lemma}
\begin{proof}
This follows from Lemma \ref{primedivisorbound} upon observing that 
\begin{align*}
|\cbr{I \in Q_{\text{prin}}(M) : J \subseteq I    }|
&=
|\cbr{I H_I : I \in Q(M/R), J \subseteq I H_I}| 
\\
&\leq 
|\cbr{I H_I : I \in Q(M/R), J \subseteq I}| 
\\
&=
|\cbr{I : I \in Q(M/R), J \subseteq I}|.
\end{align*}
\end{proof}

The proof of Lemma \ref{separation facts}(c) must be changed because distinct ideals in $Q_{\prin}(M_k)$ may not be relatively prime. 
To deal with this, we modify the definition of $\chi$ so that its support is contained in $([-3/(8R^{1/d}), -1/(8R^{1/d})] \cap [1/(8R^{1/d}), 3/(8R^{1/d})])^d$. Here is a restatement of Lemma \ref{separation facts}(c) and its new proof. 
\begin{lemma}
Let $k \in \mathbb{N}$. Suppose that $I_1, I_2 \in Q_{\prin}(M_k)$, $r_1 \in I_1^{-1}$ and $ r_2 \in I_{2}^{-1}$. 
Then $I_1 = I_1' H_{I_1'}$ and $I_2 = I_2' H_{I_2'}$ for some $I_1',I_2' \in Q(M_k/R)$.  
If $r_1 = r_2$ and $I_1 \neq I_2$, then 
$r_1 = r_2 \in H_{I_1'}^{-1} + H_{I_2'}^{-1}$
 and (consequently) 
$\dist(r_1,\supp(\chi)) = \dist(r_2,\supp(\chi)) > 3\eta_k$. 
\end{lemma}
\begin{proof}
Since $I_1 \neq I_2$, we must have $I_1' \neq I_2'$.  
Thus $I_1'$ and $I_2'$ are distinct prime ideals and (hence) are relatively prime, i.e., $1 \in I_1'+I_2'$. 
Set $r = r_1 = r_2$. 
Since $r \in I_1^{-1}$ and $r \in I_2^{-1}$, we have $r I_1' \subseteq H_{I_1'}^{-1}$ and $r I_2' \subseteq H_{I_2'}^{-1}$. 
Therefore $r = r \cdot 1 \in r(I_1' + I_2') = r I_1' + r I_2' \subseteq H_{I_1'}^{-1} + H_{I_2'}^{-1}$.
Thus, if $r \neq 0$, then 
$$
|r|^d \gtrsim N(r) \geq N(H_{I_1'}^{-1} + H_{I_2'}^{-1}) \geq R^{-1}. 
$$
Since $\supp(\chi)$ is contained in $([-3/(8R^{1/d}), -1/(8R^{1/d})] \cap [1/(8R^{1/d}), 3/(8R^{1/d})])^d$, we get the desired result. 
\end{proof} 
Note that (with the modifications to $\chi$)  
all the estimates on $F_k$, $\mu_k$, and the other functions and measures still hold with suitable modifications to the implicit constants (which now depend on $R$). 

Finally, the calculation  \eqref{I Jk norm calc} in the proof of Lemma \ref{h_2_est} must be modified. 
Now $I = I' H_{I'}$ and $J_k = J_k' H_{J_k'}$ 
for some $I' \in Q(M_k/R)$ and 
$J_k' \in Q(M_k^{1+\frac{\rho}{d}}/R)$. 
So \eqref{I Jk norm calc} becomes 
\begin{align*}
N(\delta^{-1} I \cap \delta^{-1} J_k)^{1/d}
&=
N(\delta^{-1} I' H_{I'} \cap \delta^{-1} J_k' H_{J_k'})^{1/d} 
\gtrsim 
N(I' H_{I'} \cap J_k' H_{J_k'})^{1/d}
\\
&\geq  
N(I' \cap J_k')^{1/d} 
= 
(N(I')N(J_k'))^{1/d}
\gtrsim M_k^{2+\frac{\rho}{d}}.
\end{align*}
No other changes are needed.

\section{Glossary of Fixed Constants}\label{sec:glossary}
\begin{itemize}

\item $\rho$: Real number satisfying $-d < \rho < d$. (Theorem \ref{mainthm3}, p.\pageref{mainthm3})

\item $\rho^{+} = \max\{\rho,0\}$ (Theorem \ref{mainthm3}, p.\pageref{mainthm3})

\item $\rho^{-} = \max\{-\rho,0\}$ (Theorem \ref{mainthm3}, p.\pageref{mainthm3})

\item $B = \cbr{\omega_1,\ldots,\omega_d}$ is the integral basis for $K$. 
(Section \ref{Algebraic Number Theory: Basics section}, p.\pageref{Bbasis-2})

\item $C_B = \max\cbr{ \textstyle \rbr{\sum_{i=1}^d |\tau(\omega_i)|^2}^{1/2} : \tau \in \mathcal{E}_K }$, 
where $\mathcal{E}_K$ is the set of all injective homomorphisms from $K$ into $\CC$. (Lemma \ref{C1 bounds lemma}, p.\pageref{C1 bounds lemma})

\item $\delta$: The different ideal of $\OK$. 
(Section \ref{Algebraic Number Theory: Basics section}, p.\pageref{different ideal defn})

\item $T$: $d \times d$ matrix defined by $T = [\tr(\omega_i \omega_j)]$.
(Section \ref{Algebraic Number Theory: Basics section}, p.\pageref{Tmatrix})

\item $\phi$: Non-negative Schwartz function on $\RR^d$ satisfying $\supp(\phi) \subseteq [-1,1]^d$, and 
$\widehat{\phi}(0) = \int \phi(x) dx = 1$.
(Section \ref{functions phi and Phi section}, p.\pageref{phi defn})

\item $\Phi_{I,\eta}$: 
Bump function supported on the $\eta$-neighborhood of the fractional ideal $I^{-1}$.
(Section \ref{functions phi and Phi section}, p.\pageref{Phi defn})

\item $Q(M)$: The set of prime ideals $I$ of $\OK$ with $M^d/2 \leq N(I) \leq M^d$. (Section \ref{collection of ideals section}, p.\pageref{collection of ideals section}) 

\item $(M_k)_{k=0}^{\infty}$: A rapidly increasing sequence of large, positive real numbers. (Section \ref{functions Fk section}, p.\pageref{Mk defn})

\item $J_k$: A prime ideal of $\OK$ satisfying 
$M_{k}^{d+\rho}/2 \leq N(J_k) \leq M_k^{d+\rho}$. (Section \ref{functions Fk section}, p.\pageref{Jk defn})

\item $P_k$: 
$$
\P_k = 
\left\{
\begin{array}{ll}
M_k^{\rho^-}
& \text{if } k \text{ is odd}, \\
0 & \text{if } k \text{ is even}. \\
\end{array}
\right.
$$
(Section \ref{functions Fk section},p.\pageref{Pk defn})

\item $c_k$: 
$$
\dfrac{1}{c_k}  
= 
\displaystyle \P_k N(J_k) + \sum_{I \in Q(M_k)}  N(I). 
$$
(Section \ref{functions Fk section},p.\pageref{ck defn})

\item $\tau$: Real number satisfying $\tau > 1$. 
(Section \ref{functions Fk section},p.\pageref{etak defn})

\item $\eta_k = M_k^{-(1_\tau)}$ (Section \ref{functions Fk section},p.\pageref{etak defn})

\item $F_k$: 
$$F_k(x) 
= c_k 
\rbr{ 
\P_k \Phi_{J_k,\eta_k}(x) + \sum_{I \in Q(M_k)} \Phi_{I,\eta_k}(x)  
}.
$$
(Section \ref{functions Fk section},p.\pageref{Fk defn})

\item $
C_0 = 2^{-1/d}C_B^{-1} \|T^{-1}\|^{-1} N(\delta)^{1/d}. 
$
(Section \ref{functions Fk section}, p.\pageref{C0 defn})

\item $\chi:$ Non-negative Schwartz function on $\RR^d$  with support contained in  $([-3/8,-1/8] \cap [1/8,3/8])^d$ and $\widehat{\chi}(0) = \int \chi(x) dx = 1$. 
(Section \ref{measure muk section}, p.\pageref{chi defn})

\item $\mu_k$: Measures on $\RR^d$ defined by 
$$
d\mu_0 = \chi dx, \quad d\mu_k = \chi  F_{1} \cdots F_{{k}} dx \quad \text{ for all } k \in \NN.
$$
(Section \ref{measure muk section}, p.\pageref{muk defn})

\item $N$: An integer with $N \geq 100d$. 
(Section \ref{measure muk section}, p.\pageref{N defn})

\item $\mu_{l,k}$:  Measures on $\RR^d$ defined, $l \geq k \geq 1$, by the formula $$\mu_{l,k} 
:=  
F_{l} \cdots F_{{k+1}} F_{{k-1}} \cdots F_{1} \chi
= F_{l} \cdots F_{{k+1}} \mu_{k-1}. 
$$
Note $\mu_{k,k} := \mu_{k-1}$. 
(Section \ref{measures mulk section}, p.\pageref{mulk defn})

\item $c_{\phi}$: Real number satisfying $0 < c_{\phi} < 1$ such that $\widehat{\phi^2}(\xi) \geq c_{\phi}$ for all $|\xi| \leq c_{\phi}$.
(Section \ref{Pointwise Lower Bound}, p.\pageref{cphi defn})

\end{itemize}

\bibliographystyle{myplain}
\bibliography{Approximation_In_Number_Fields}
\end{document}